\documentclass{amsart}

\usepackage{amsmath}
\usepackage{amssymb}
\usepackage{indentfirst}
\usepackage{galois}
\usepackage{amsthm}
\usepackage{amsfonts}
\usepackage{CJK}

\theoremstyle{proclaim}
\newtheorem{theorem}{Theorem}[section]
\newtheorem{lemma}[theorem]{Lemma}
\newtheorem{corollary}[theorem]{Corollary}
\newtheorem{proposition}[theorem]{Proposition}
\newtheorem{claim}[theorem]{Claim}
\theoremstyle{fancyproclaim}

\theoremstyle{definition}
\newtheorem{remark}[theorem]{Remark}
\newtheorem{definition}[theorem]{Definition}
\newtheorem{example}[theorem]{Example}
\theoremstyle{fancystatement}

\numberwithin{equation}{section}

\begin{document}

\title[Spectra of hyperbolic weighted composition semigroups]{Spectra of generators of hyperbolic composition and weighted composition semigroups}

\author{Yong-Xin Gao}
\address{Yong-Xin Gao\\
School of Mathematical Sciences and LPMC, Nankai University, Tianjin 300071, P.R. China.}
\email{ygao@nankai.edu.cn}

\author{Ze-Hua Zhou}
\address{Ze-Hua Zhou\\
School of Mathematics, Tianjin University, Tianjin 300354, P.R. China.}
\email{zhzhou@tju.edu.cn}

\keywords{weighted composition operators, semigroups, spectra}
\subjclass[2020]{Primary 47A10; Secondary 47B33, 47D03}

\date{}
\thanks{This work is supported by the National Natural Science Foundation of China (Grant Nos. 12571088, 12271396).}
\begin{abstract}
In this paper, we provide complete characterizations for the spectrum, essential spectrum, and point spectrum of the generators of weighted composition $C_0$-semigroups induced by hyperbolic semiflows on Bergman spaces. We give an explicit example showing that the spectral properties can be influenced by the behavior of the semigroup near non-fixed self-contact points.
\end{abstract}

\maketitle

\section{Introduction}\label{sec1}

A \emph{(one-parameter) semigroup} of bounded operators on a Banach space $X$ is a semigroup homomorphism $t\mapsto T_t$ from the additive semigroup $[0,+\infty)$ into the algebra of bounded operators on $X$, that is, $T_0 = Id_X$ and $T_{s+t} = T_sT_t$ for all $s,t\geq 0$. A semigroup $(T_t)_{t\geq0}$ is called \emph{strongly continuous} (or a $C_0$-semigroup) if the orbit map $t\mapsto T_tx$ is continuous from $[0,+\infty)$ into $X$ for each $x \in X$. 

For any $C_0$-semigroup $(T_t)$, its \emph{generator} $A : D(A) \subset X \to X$ is defined by
$$Ax=\lim_{t\to 0}\frac{1}{t}(T_tx-x)$$
for every $x \in D(A)$. The generator of a $C_0$-semigroup is always a closed and densely defined linear operator on $X$ (see Theorem 1.4 of Chapter \uppercase\expandafter{\romannumeral2} in \cite{equ}).

The spectra of the operators in a $C_0$-semigroup are closely related to the spectrum of the generator. In fact, one always has the inclusion
\begin{align}\label{SMT0}
\sigma(T_t) \supset  \mathrm{e}^{t\sigma(A)}
\end{align}
for any  $C_0$-semigroup $(T_t)$. For a general $C_0$-semigroup this inclusion is not always an equality. A detailed account of the spectral theory of $C_0$-semigroups and their generators can be found in \cite{equ}. 

In this present paper, we focus on $C_0$-semigroups of composition operators and weighted composition operators on Banach spaces of analytic functions on the unit disk $\mathbb{D} = \{z \in \mathbb{C} : |z| < 1\}$. Recall that a \emph{weighted composition operator} $uC_\varphi$ is defined as
$$uC_\varphi (f) =u  f \circ\varphi$$
for $f$ analytic on $\mathbb{D}$, where $u$ is an analytic function on $\mathbb{D}$ and $\varphi$ is an analytic self-map of $\mathbb{D}$. When $u$ is identically $1$ on $\mathbb{D}$, one obtains the \emph{composition operator} $C_\varphi$. The study of composition and weighted composition operators can be found in numerous references, including the early classical works \cite{COW, FOR, NOR, SHA}.

The spectra of single weighted composition operators on various types of Banach spaces of analytic functions has been extensively studied in numerous papers.  The most complete descriptions are usually obtained when the operator is compact or invertible., see \cite{sp4,sp1,sp2,sp3} for instance.  Yet it is not very clear when a weighted composition operator $uC_\varphi$ can be embedded into a $C_0$-semigroup. However, when such an embedding exists, the semigroup structure offers another perspectives for computing the spectrum of $uC_\varphi$. And in this context the spectrum of the generator of such a weighted composition semigroup plays a crucial role.

In the very recent papers \cite{aim,jfa}, the authors investigate the spectra of the generator of $(u_tC_{\varphi_t})_{t\in\mathbb{R}}$ when it forms a $C_0$-\emph{group}. Note that a semigroup $(u_tC_{\varphi_t})_{t\geq 0}$ extends to a group precisely when each operator $u_tC_{\varphi_t}$ is invertible for $t>0$, and this forces every $\varphi_t$ to be an automorphisms of the disk $\mathbb{D}$. The authors in \cite{aim} determine the spectrum of the generator of such a $C_0$-groups in the case that $(\varphi_t)$ is a flow of hyperbolic automorphisms. These result generalize known descriptions for the spectra of invertible weighted composition operators to wider classes of analytic function spaces.

Given these developments, it is natural and of significant interest to ask about the spectral properties of a general $C_0$-semigroup $(u_tC_{\varphi_t})_{t\geq 0}$ that does not extend to a group. This is exactly the main purpose of the present paper. In this setting, the $(\varphi_t)$ no longer consists of automorphisms, hence a finer analysis is evidently required.

This paper contains two main results. First, Theorem \ref{main} gives a complete description of the spectrum of the generator of a hyperbolic weighted composition semigroup under the broadest possible hypotheses. In addition, we also characterize the essential spectrum of such generators, which is presented in Theorem \ref{ess}.

As a special case, we obtain a full spectral characterization of the generators of (unweighted) composition semigroups. The eigenvalues of such generators have been studied in many literatures. In \cite{tezheng1, tezheng2} the author gives a complete description of the eigenvalues of the generator of a hyperbolic composition semigroup. However, a complete picture of the whole spectrum of the generator of a hyperbolic composition semigroup has been lacking, and our Corollary \ref{compositions} fills this gap. 

The paper is organized as follows. We first introduce the settings and fix our notations in Section \ref{sec2}.  Then in Section \ref{bu} we state our main results, provide explanatory remarks, and compare our results with those previously known in the literatures. Section \ref{sec4} discusses the properties of semiflows and semicocycles that induced the weighted composition semigroups. In Section \ref{sec3}, we give the spectral radii of the weighted composition operators in the semigroups. And Section \ref{sec5} determines the eigenvalues of the generators. Section \ref{sec6} contains complete proofs of our main results Theorem \ref{main} and Theorem \ref{ess}. Finally, in Section \ref{further} we give some further remarks and questions.

\section{Preliminaries}\label{sec2}

The notations and terminology introduced in this section will be used consistently throughout the paper.

The space being considered in this paper is the Bergman spaces. For a (bounded) domain $\Omega\subset\mathbb{C}$ and $p\geq 1$, the Bergman space $A^p(\Omega)$ consists of all analytic functions on $\Omega$ for which
$$\|f\|_p=\left(\int_\Omega|f(z)|^p\mathrm{d}V(z)\right)^{1/p}<\infty,$$
where $\mathrm{d}V$ denotes the Lebesgue area measure. For brevity we write $A^p$ for short of $A^p(\mathbb{D})$, the Bergman space on the unit disk. For a boundary point $\zeta\in\partial \mathbb{D}$, we denote by $A^p_\zeta$ the set of functions $f$  that belong to $A^p(U\cap\mathbb{D})$ for some neighborhood $U$ of $\zeta$.

If $(u_tC_{\varphi_t})_{t\geq0}$ is a  $C_0$-semigroup of bounded operators on $A^p$, necessarily $(\varphi_t)$ is a continuous semigroup (also called a \emph{semiflow}) of holomorphic self-maps on $\mathbb{D}$, that is,
\begin{enumerate}
\item $\varphi_0$ is the identity in $\mathbb{D}$;
\item $\varphi_{t+s} = \varphi_t \circ\varphi_s$ for all $t,s \geq 0$;
\item the map $t \mapsto \varphi_t(z)$ is continuous on $[0,+\infty)$ for every $z \in\mathbb{D}$.
\end{enumerate}
Correspondingly, $(u_t)$ must be a continuous  \emph{semicocycle} for  $(\varphi_t)$, i.e.,
\begin{enumerate}
\item $u_0 \equiv 1$ on $\mathbb{D}$;
\item $u_{t+s}=u_t  \cdot u_s\circ\varphi_t$ for all $t,s\geq0$;
\item the map $t \mapsto u_t(z)$ is continuous on $[0,+\infty)$ for every $z \in \mathbb{D}$.
\end{enumerate}

It is not hard to prove that all iterates $\varphi_t$ for $t>0$ share a common Denjoy-Wolff point in $\overline{\mathbb{D}}$ (see Theorem 8.3.1 in \cite{cs}), which is denoted by $\zeta_0$ throughout this paper. If $\zeta_0\in \partial \mathbb{D}$, then there exists $\alpha_0\geq0$ such that the angular derivative of $\varphi_t$ at $\zeta_0$ is
$$\varphi_t'(\zeta_0)= \mathrm{e}^{-\alpha_0t}.$$
The number $\alpha_0$ is called the \emph{spectral value} of  $(\varphi_t)$ at $\zeta_0$. A semiflow  $(\varphi_t)$ is called \emph{hyperbolic} if it has positive spectral value at a boundary Denjoy-Wolff point.

For a hyperbolic semiflow  $(\varphi_t)$, by Theorem 12.1.4 in \cite{cs}, all maps $\varphi_t$ for $t>0$ share the same set of fixed points. And at each such fixed point $\zeta$ (which necessarily lies on $\partial\mathbb{D}$) other than the Denjoy-Wolff point, there exists $\alpha<0$ such that 
$$\varphi_t'(\zeta)= \mathrm{e}^{-\alpha t}.$$
for all $t\geq0$. We call $\alpha$ the spectral value of $(\varphi_t)$ at $\zeta$.

Unfortunately, the spectral properties of a weighted composition semigroup $(u_tC_{\varphi_t})$ cannot always be determined by the behavior of  $(\varphi_t)$ and $(u_t)$ near the fixed points of $(\varphi_t)$, even when they exhibit excellent behavior near those points. We will present a concrete examples in Section \ref{further} to illustrate this fact, see Example \ref{yigeyue}. Nevertheless, we are therefore led to consider points with the following property.

\begin{definition}
A point $\zeta\in\partial D$ is called a \emph{self-contact point} of a hyperbolic semiflow $(\varphi_t)$ if there exists $t>0$ such that $\varphi_t(U\cap\mathbb{D})\cap U\ne\emptyset$ for any neighborhood $U$ of $\zeta$.
\end{definition}

We say that $(\varphi_t)$ is analytic at a self-contact point $\zeta$ if $\varphi_t$ is analytic at $\zeta$ for some $t>0$. In Section \ref{sec3} we will show that this condition actually implies that $\varphi_t$ is analytic at $\zeta$ for every $t\geq0$.

Throughout the paper, we work under the following hypotheses on the $C_0$-semigroup $(u_tC_{\varphi_t})$: we require that $(\varphi_t)$ is analytic at its self-contact points, and that the semicocycle $(u_t)$ is continuous at these points.

Note that under these assumptions, $(\varphi_t)$ has finitely many self-contact points, and each of them is actually a fixed point of  $(\varphi_t)$. Denote these points by $\{\zeta_j\}_{j=0}^m$, where $\zeta_0$ is the Denjoy-Wolff point. By the continuity of $(u_t)$ at these points we have
\begin{align*}
u_{t+s}(\zeta_j) &= \lim_{z\to\zeta_j}u_{t+s}(z)\\
&=\lim_{z\to\zeta_j}u_t(z)u_s(\varphi_t(z))\\
&=u_t(\zeta_j)u_s(\zeta_j)
\end{align*}
Note that $u_t(\zeta_j)=\lim_{k\to\infty} u_t((1-1/k)\zeta_j)$, which is measurable with respect to $t\in[0,+\infty)$. So the Cauchy equation theory, see Proposition 8.1.14 in \cite{cs} for instance, shows that there exists $\beta_j\in\mathbb{C}\cup\{-\infty\}$ such that 
\begin{align}\label{valuebeta}
u_t(\zeta_j)= \mathrm{e}^{-\beta_j t}
\end{align}
for $t>0$. 

For the rest of this paper, we consistently write 
\begin{align}\label{gamma}
\gamma_j=2\alpha_j/p+ \operatorname{Re}\beta_j,
\end{align}
where $\alpha_j$ is the spectral value of $(\varphi_t)$ at the fixed point $\zeta_j$, and $\beta_j$ is given as in \eqref{valuebeta}. For indices $j\geq m+1$ we simply set $\gamma_j=-\infty$. Thus by relabeling, we may always assume that the sequence $\{\gamma_j\}$ is decreasing for $j\geq 1$.

Finally, for any pair of numbers $a\leq b$ in $\mathbb{R}\cup\{\pm\infty\}$, we denote by $\lfloor a,b\rceil$ the vertical strip $\{w \in\mathbb{C} : a \leq  \operatorname{Re}w\leq b\}$, and by $i\lfloor a,b\rceil$ the horizontal strip $\{w \in\mathbb{C} : a \leq  \operatorname{Im}w\leq b\}$.

\section{Main results}\label{bu}

With the preparations in Section \ref{sec2}, we can now state the main results of the paper.

\begin{theorem}\label{main}
Let $(u_tC_{\varphi_t})$ be a $C_0$-semigroup on $A^p$, where $(\varphi_t)$ is a hyperbolic semiflow. Suppose $(\varphi_t)$ is analytic at its self-contact points, and $(u_t)$ is continuous at these points. Then the spectrum of the generator $A$ of $(u_tC_{\varphi_t})$ is given as follows:
\begin{enumerate}
\item If $\gamma_0 \geq \gamma_1$, then
$$\sigma(A) = \lfloor-\infty, \gamma_2\rceil\cup\lfloor\gamma_1,\gamma_0\rceil.$$

\item If $\gamma_2 < \gamma_0 < \gamma_1$, then
$$\sigma(A) = \lfloor-\infty,  \gamma_2\rceil\cup\lfloor\gamma_0 ,\gamma_1\rceil.$$

\item If $\gamma_0 \leq \gamma_2$, then
$$\sigma(A) = \lfloor-\infty, \gamma_1\rceil.$$
\end{enumerate}
Here the numbers $\gamma_j$ are defined in \eqref{gamma} and are decreasing for $j\geq 1$.
\end{theorem}

\begin{remark}
It should be pointed out that we never require $(u_t)$ to be bounded. Indeed, if one is willing to impose the additional assumption that $u_t$ is bounded on $\mathbb{D}$ for all $t\geq0$, several arguments in this paper could be simplified considerably.
\end{remark}

The next theorem describes the essential spectrum of the generators.

\begin{theorem}\label{ess}
Let $(u_tC_{\varphi_t})$ be a $C_0$-semigroup on $A^p$, where $(\varphi_t)$ is a hyperbolic semiflow. Suppose $(\varphi_t)$ is analytic at its self-contact points, and $(u_t)$ is continuous at these points. If both $\gamma_0$ and $\gamma_1$ are not $-\infty$, then the essential spectrum of the generator $A$ of $(u_tC_{\varphi_t})$ is
$$\sigma_e(A) =\bigcup_{\gamma_j\ne-\infty}\{\lambda\in\mathbb{C} : \operatorname{Re}\lambda=\gamma_j\}.$$
\end{theorem}

We highlight two important special cases of Theorem \ref{main} and Theorem \ref{ess}. First, if $(u_tC_{\varphi_t})$ forms a group, then each operator $u_tC_{\varphi_t}$ is invertible. By Theorem 3.6 in \cite{sp2},  every $\varphi_t$ is a hyperbolic automorphism, whence we have $\gamma_2=-\infty$. Consequently, Theorem \ref{main} gives
$$\sigma(A)=\Bigl\lfloor \min\{\gamma_0,\gamma_1\}, \max\{\gamma_0,\gamma_1\}\Bigr\rceil,$$
and Theorem \ref{main} shows that $\sigma_e(A)$ is the boundary of this strip. These are exactly the results for hyperbolic weighted compositon groups on Bergman spaces obtained in \cite{aim} (see Theorem 6.8 therein).

The second important situation arises by taking $u_t\equiv 1$ and this yields a characterization of the spectrum for hyperbolic composition semigroups. According to \cite{strongly}, each continuous semiflow is automatically strongly continuous on $A^p$, and therefore induces a $C_0$-semigroup of composition operators.

\begin{corollary}\label{compositions}
Let $(\varphi_t)$ be a hyperbolic semiflow of analytic self-maps of $\mathbb{D}$. Suppose $(\varphi_t)$ is analytic at its self-contact points. Then the spectrum of the generator $A$ of $C_0$-semigroup $(C_{\varphi_t})$ on $A^p$ is given as follows:
\begin{enumerate}
\item If  $\zeta_0$ is the only fixed point of $(\varphi_t)$, then 
$$\sigma(A)=\left\lfloor -\infty, \frac{2\alpha_0}{p}\right\rceil.$$
And the point spectrum of $A$ is exactly the interior of this half-plane.

\item If $(\varphi_t)$ has two fixed points, then 
$$\sigma(A)=\left\lfloor \frac{2\alpha_1}{p}, \frac{2\alpha_0}{p}\right\rceil.$$
And the point spectrum of $A$ is exactly the interior of this strip.

\item If $(\varphi_t)$ has more than two fixed points, then 
$$\sigma(A)=\left\lfloor -\infty, \frac{2\alpha_2}{p}\right\rceil\cup\left\lfloor \frac{2\alpha_1}{p}, \frac{2\alpha_0}{p}\right\rceil.$$
And the point spectrum of $A$ is exactly the interior of the strip $\lfloor \frac{2\alpha_1}{p}, \frac{2\alpha_0}{p}\rceil$.
\end{enumerate}

\end{corollary}

The conclusions about the spectrum and the essential spectrum in Corollary \ref{compositions} follow directly from Theorem \ref{main} and Theorem \ref{ess}. The statements concerning the eigenvalues will be obtained from  \eqref{difsolu} together with Corollary \ref{zhishu}, see Remark \ref{2yue25} for specific instructions. We note that the eigenvalues of generators of hyperbolic composition semigroups have already been characterized under more general settings in \cite{tezheng1,tezheng2}. However,  a complete description of the full spectrum of such generators was lacking before Corollary \ref{compositions}.

The proofs of Theorem \ref{main} and Theorem \ref{ess} will be given in Section \ref{sec6}.

\section{Semiflows and semicocycles}\label{sec4}

In this section, we conduct a detailed investigation of the behavior of semiflows and semicocycles near the boundary, particularly in the vicinity of the fixed points of semiflows.

Since discussing the properties of semiflows directly on the disk $\mathbb{D}$ is sometimes inconvenient, we occasionally employ alternative models. For a hyperbolic semiflow $(\varphi_t)$, there exists a univalent map $h$ on $\mathbb{D}$ such that
$$h\circ\varphi_t(z)=h(z)+t$$
for all $z\in\mathbb{D}$ and $t\geq 0$. For convenience, we may assume $h(0) = 0$, then with this normalization $h$ is unique. Set $\Omega =h(\mathbb{D})$. Then the pair $(\Omega, h)$ is called a \emph{canonical model} for $(\varphi_t)$, and $h$ is referred to as the  \emph{Koenigs function} of  $(\varphi_t)$. According to Theorem 9.3.5 in \cite{cs}, the domain $\Omega$ is contained in a horizontal strip of width $\pi/\alpha_0$.

The following theorem shows that our requirement on the analyticity of $(\varphi_t)$ at its self-contact points, which is imposed in Section \ref{sec2},  is equivalent to the condition that $\varphi_t$ is analytic at the  self-contact points for all $t>0$. Moreover, one can see that this condition is intimately related to the properties of the Koenigs function $h$ near the fixed points of $(\varphi_t)$.

\begin{theorem}\label{ana}
Let $(\varphi_t)$ be a hyperbolic semiflow of analytic self-maps of $\mathbb{D}$. If $\varphi_t$ is analytic at a boundary self-contact point $\zeta$ for some $t > 0$, then $\varphi_t$ is analytic at  $\zeta$ for all $t > 0$. Furthermore, in this case the function $\mathrm{e}^{-\alpha h}$ is analytic at $\zeta$, where $\alpha$ is the spectral value of $(\varphi_t)$ at $\zeta$.
\end{theorem}

\begin{proof}
Assume that $\varphi_{T}$ is analytic at $\zeta_0$ for some $T>0$, where $\zeta_0$ is the Denjoy-Wolff point of $(\varphi_t)$. Since $\varphi_T'(\zeta_0)<1$, there exists a neighborhood $U$ of $\zeta_0$ such that $\varphi_T(U) \subset U$. By Theorem 2.53 in \cite{Cowen}, there exists a holomorphic function $\Psi$ on $U$ such that 
$$\Psi\circ\varphi_T= \mathrm{e}^{-T\alpha_0}\cdot\Psi$$
on $U$. On the other side, we have
$$\mathrm{e}^{-\alpha_0h}\circ\varphi_T= \mathrm{e}^{-T\alpha_0}\cdot  \mathrm{e}^{-\alpha_0h}$$
on $\mathbb{D}$. Thus, the uniqueness of the Koenigs function up to a multiplicative constant implies that $\Psi=c\cdot\mathrm{e}^{-\alpha_0h}$ on $U\cap\mathbb{D}$ for some nonzero constant $c$. Therefore, the function $ \mathrm{e}^{-\alpha_0h}$ is analytic at $\zeta_0$.

Moreover, since $\Psi(\partial\mathbb{D}\cap U)$ is an analytic arc passing through $0$, and for any $t>0$,
$$\Psi\circ\varphi_t(\partial\mathbb{D}\cap U)= \mathrm{e}^{-\alpha_0t}\cdot\Psi(\partial\mathbb{D}\cap U)$$
near $\zeta_0,$ which is also an analytic arc passing through $0$. Thus $\Psi\circ\varphi_t\circ\Psi^{-1}$ is analytic at $0$ , which implies that $\varphi_t$ is analytic at $\zeta_0$ for any $t>0$.

Now let $\zeta\ne\zeta_0$ be another self-contact point of $(\varphi_t)$ and suppose $\varphi_T$ is analytic at $\zeta$ for some $T>0$. First note that $\zeta$ must be a fixed point of  $(\varphi_t)$. Since $\varphi'_T(\zeta)>1$, there exists a neighborhood $V$ of $\zeta$ where $\varphi^{-1}_T$ is analytic and $\varphi_T^{-1}(V)\subset V$. Then an identical argument shows that $ \mathrm{e}^{-\alpha h}$ is analytic at $\zeta$. And consequently, $\varphi_t$ is analytic at $\zeta$ for all $t>0$.
\end{proof}

\begin{remark}\label{remarkchang}
We give a few remarks regarding Theorem \ref{ana} and introduce several notations that will be used throughout the paper. First, following Definition 13.4.1 in \cite{cs}, the connected components of the interior of the maximal backward invariant set of $(\varphi_t)$ are called the  \emph{petals} of $(\varphi_t)$. Suppose that the self-contact points of $(\varphi_t)$ besides the Denjoy-Wolff point are $\zeta_1,\dots,\zeta_m$, and $(\varphi_t)$ is analytic at these fixed points, then by Propositions 13.4.10 and 13.4.12 in \cite{cs}, there are exactly $m$ petals, denoted throughout this paper by $\{\Delta_j\}_{j=1}^m$, such that $\zeta_j$ lies on the boundary of $\Delta_j$ for each $j = 1,2,\dots,m$.  Theorem  \ref{ana} then shows that the analyticity of $(\varphi_t)$ at $\zeta_j$ makes $\Delta_j$ conformal at $\zeta_j$ with respect to $\mathbb{D}$. In fact, by Theorem 13.5.5 in \cite{cs}, $h(\Delta_j)$ is a horizontal strip of width $\frac{\pi}{-\alpha_j}$, where $\alpha_j$ is the spectral value of $(\varphi_t)$ at $\zeta_j$. Therefore, $\mathrm{e}^{-\alpha_j h}$ gives a biholomorphic map from $\Delta_j$ onto a half-plane. This map is analytic at $\zeta_j$ and has non-zero derivative there.

Secondly,since $\mathrm{e}^{-\alpha_0 h}$ is analytic at the Denjoy-Wolff point $\zeta_0$, a sequence $\{z_n\}\subset\mathbb{D}$ converges to $\zeta_0$ if and only if $\operatorname{Re}h(z_n)\to+\infty$.  Now, for a sufficiently large $M>0$, set 
$$\Gamma_0=\Gamma_0(M)=\{w\in\Omega : \operatorname{Re}w\geq M\},$$  
and let $Q_0(x)$ and $q_0(x)$ be the supremum and infimum of the set $\{y : x+iy\in \Gamma_0\}$ respectively. Then the analyticity of $\mathrm{e}^{-\alpha_0h}$ at $\zeta_0$ shows that $Q_0(x)$, as well as $q_0(x)$, is either constant or strictly monotonic for $x$ large enough. 

Similarly, by noting that for a sequence $z_n\in\mathbb{D}$, each cluster point of this sequence is a self-contact point of $(\varphi_t)$ if and only if
$$\lim_{n\to\infty}\operatorname{Im}h(z_n)=-\infty,$$
we can take $M$ large enough such that $\{w\in\Omega : \operatorname{Re}w\leq -M\}$ has $m$ connected components, denoted by 
$$\Gamma_j=\Gamma_j(M)$$
for $j=1, 2, \dots, m$, with $\zeta_j$ lying on the boundary of each $h^{-1}(\Gamma_j)$. Define $Q_j(x)$ and $q_j(x)$ as the supremum and infimum of the set $\{y : x+iy\in \Gamma_j\}$ respectively, then one also finds that $Q_j(x)$ and $q_j(x)$ are either constants or strictly monotonic when $x$ is close to $-\infty$.  

Finally, the notation introduced above satisfies the following relations: $h(\Delta_j)$ is precisely the interior of the strip $i\lfloor q_j(-\infty),Q_j(-\infty)\rceil$, and consequently, for $j\geq 1$ we have $$Q_j(-\infty)-q_j(-\infty)=\frac{\pi}{-\alpha_j}.$$
See Theorem 13.5.5 and 13.5.6 in \cite{cs} for instantce.
\end{remark}

Theorem \ref{ana}  also gives the following useful corollary. Recall that 

\begin{corollary}\label{zhishu}
Let $(\varphi_t)$ be a hyperbolic semiflow and $h$ be its Koenigs function. Suppose $(\varphi_t)$ is analytic at a boundary self-contact point $\zeta$.
\begin{enumerate}
\item If  $\zeta=\zeta_0$ is the Denjoy-Wolff point, then $\mathrm{e}^{\mu h}\in A^p_{\zeta_0}$ if and only if $ \operatorname{Re}\mu<2\alpha_0/p$;

\item If $\zeta\ne\zeta_0$, then $ \mathrm{e}^{\mu h}\in A^p_{\zeta}$  if and only if $ \operatorname{Re}\mu>2\alpha/p$, where $\alpha$ is the spectral value of $(\varphi_t)$ at $\zeta$.
\end{enumerate}
\end{corollary}

\begin{proof}
Because $\mathrm{e}^{-\alpha h}$ is analytic at $\zeta$, it is equivalent to $z-\zeta$ near $\zeta$. Hence $ \mathrm{e}^{\mu h}\in A^p_{\zeta}$ if and only if $(z-\zeta)^{-\mu/\alpha}\in A^p_{\zeta}$. And then assertions (i) and (ii) follows directly from Lemma 7.3 in \cite{Cowen}.
\end{proof}

Another important tool for studying the properties of a semiflow is its infinitesimal generator. For a semiflow $(\varphi_t)$ on the disk, there exists a unique holomorphic function $G$ on $\mathbb{D}$ such that
$$\frac{\partial\varphi_t(z)}{\partial t}=G(\varphi_t(z)).$$
The function $G$ is called the \emph{infinitesimal generator} of $(\varphi_t)$. For a hyperbolic semiflow $(\varphi_t)$ one has $G=1/h'$, see Theorem 10.1.4 in \cite{cs}.

We now turn to the semicocycles. If $(u_tC_{\varphi_t})$ is a $C_0$-semigroup, the proof of Theorem 1 in \cite{Wolf} shows that the map $t \mapsto u_t(z)$ is differentiable on $[0,+\infty)$ for every $z \in \mathbb{D}$, and the function
\begin{align}\label{gene}
g(z) =\frac{\partial}{\partial t} u_t(z)\Big|_{t=0}
\end{align}
is holomorphic on $\mathbb{D}$. We call $g$ the \emph{generator} of $(u_t)$. Solving equation \eqref{gene} gives 
$$u_t(z)=\exp\left(\int_0^tg\circ\varphi_s(z)\mathrm{d}s\right).$$
If we set
\begin{align}\label{boundary}
v(z)=\exp\left(\int_0^z\frac{g(\xi)}{G(\xi)}\mathrm{d}\xi\right),
\end{align}
then a straightforward verification yields
$$u_t(z)=\frac{v\circ\varphi_t(z)}{v(z)}.$$
for all $z\in\mathbb{D}$ and $t\geq0$. The function $v$ is called the \emph{semicoboundary} associated with $(u_t)$.

For the remainder of this section, we derive estimates for the growth of the function $v$ near the self-contact points of $(\varphi_t)$. Keep in mind that the following notation and setting are adopted for the rest of the section: the self-contact points of $(\varphi_t)$ are denoted by $\{\zeta_j\}_{j=0}^m$ where $m\geq 0$, and $\zeta_0$ is the Denjoy-Wolff point; the semicocycle $(u_t)$ is continuous at these points, and the sequence $\{\gamma_j\}$ is defined as in \eqref{gamma}.

Note that we have never required the functions in the semicocycle $(u_t)$ to be bounded, therefore we begin with the following lemma.

\begin{lemma}\label{commombd}
Suppose $\beta_0\ne-\infty$. Then for any $T > 0$, there exist a neighborhood $U$ of $\zeta_0$ and a constant $C=C(T)>1$ such that $$1/C \leq |u_t(z)| \leq C$$ for all $z\in U\cap\mathbb{D}$ and $t \in[0,T]$. 
\end{lemma}

\begin{proof}
Since $\beta_0\ne-\infty$, we have $u_T \in (H^\infty(U\cap\mathbb{D})) ^{-1}$ for some neighborhood  $U$ of $\zeta_0$. We may assume that $U\cap\mathbb{D}$ is invariant under $(\varphi_t)$. Then  $u_{nT} \in (H^\infty( U\cap\mathbb{D})) ^{-1}$ for every positive integer $n$.

Fix any $t \in (0,1]$. There also exists a neighborhood $V\subset U$ of $\zeta_0$ such that $V\cap\mathbb{D}$ is invariant under $(\varphi_t)$, and $u_t \in (H^\infty( V\cap\mathbb{D})) ^{-1}$. We may assume that all other fixed points of $(\varphi_t)$ other than $\zeta_0$ lie outside the closure of $U\cap\mathbb{D}$, then $\varphi_t(z)$ converges to $\zeta_0$ as $t \to+\infty$ uniformly on $U\cap\mathbb{D}$. Therefore, we can take a positive integer $N$ large enough so that $\varphi_{NT}(U) \subset V$. Then for any $z\in U\cap\mathbb{D}$,
$$u_t(z)=\frac{u_{NT}(z)}{u_{NT}(\varphi_t(z))}  u_t(\varphi_{NT}(z)),$$
which shows that $u_t \in  (H^\infty( U\cap\mathbb{D})) ^{-1}$ for every $t \in [0,T]$, and hence for every $t\geq 0$. In other words, both $(u_t )$ and $(1/u_t )$ are bounded cocycles for $(\varphi_t)$ on  $U\cap\mathbb{D}$. Then this lemma now follows from Lemma 2.1 in \cite{Wolf}.
\end{proof}

\begin{proposition}\label{4zc}
Suppose $\beta_0\ne-\infty$. Let $v$ be the semicoboundary defined as in \eqref{boundary}. Then for any $\epsilon>0$, there exists $M > 0$ and $C,c > 0$ such that
$$c\left| \mathrm{e}^{h(z)}\right|^{ \operatorname{Re}\beta_0-\epsilon}<|v(z)|<C\left| \mathrm{e}^{h(z)}\right|^{ \operatorname{Re}\beta_0+\epsilon}$$
for all $z \in \mathbb{D}$ with $ \operatorname{Re}h(z) \geq M$.
\end{proposition}

\begin{proof}
Define
$$\tilde{g}=g\circ h^{-1}-\beta_0$$
and
$$\tilde{v}(w) = v \circ h^{-1}(w) \cdot \mathrm{e}^{-\beta_0w} = \exp\int_0^w\tilde{g}(\eta) \mathrm{d}\eta$$
for $w\in\Omega$. Then it suffices to prove that for all $w\in\Omega$ with $\operatorname{Re}w\geq M$,
\begin{align}\label{0}
c \mathrm{e}^{-\epsilon \operatorname{Re}w}<|\tilde{v}(w)|<C \mathrm{e}^{\epsilon \operatorname{Re}w}.
\end{align}

Observe that
$$\log u_1\circ h^{-1}(w)=\beta_0+\int_w^{w+1}\tilde{g}(\eta) \mathrm{d}\eta,$$
which tends to $\beta_0$ as $ \operatorname{Re}w\to+\infty$. So for any $\epsilon>0$, we can take $M>0$ such that
\begin{align}\label{1}
\left|\int_w^{w+1}\tilde{g}(\eta) \mathrm{d}\eta\right|<\epsilon
\end{align}
when $ \operatorname{Re}w\geq M$. Moreover, by Lemma \ref{commombd} we may also assume that
\begin{align}\label{2}
-C< \operatorname{Re}\int_w^{w+t}\tilde{g}(\eta) \mathrm{d}\eta<C
\end{align}
for some $C>0$ whenever $\operatorname{Re}w\geq M$ and $t\in[0,1]$.

Fix an arbitrary $w_0=x_0+iy_0\in \Omega$ with $x_0\geq M$. From \eqref{1} we obtain
$$\left|\int\int_{[0,1]\times[0,y_0]} \operatorname{Im}\tilde{g}(x_0+t+iy) \mathrm{d}t \mathrm{d}y\right|<\frac{\pi}{\alpha_0}\epsilon.$$
By Fubini's Theorem, there exists $t_0\in[0,1)$ such that
$$\left|\int_0^{y_0} \operatorname{Im}\tilde{g}(x_0+t_0+iy) \mathrm{d}y\right|<\frac{\pi}{\alpha_0}\epsilon,$$
i.e.,
$$\left| \operatorname{Re}\int_{x_0+t_0}^{w_0+t_0}\tilde{g}(\eta) \mathrm{d}\eta\right|<\frac{\pi}{\alpha_0}\epsilon.$$
Consequently,
\begin{align}\label{3}
\left| \operatorname{Re}\int_{x_0}^{w_0}\tilde{g}(\eta) \mathrm{d}\eta\right|&=\left| \operatorname{Re}\left(\int_{x_0}^{x_0+t_0}+\int_{x_0+t_0}^{w_0+t_0}-\int_{w_0}^{w_0+t_0}\right)\tilde{g}(\eta) \mathrm{d}\eta\right|\nonumber\\
&<2C+\frac{\pi}{\alpha_0}\epsilon.
\end{align}

Write $x_0=M+k+r$ with a positive integer $k$ and $r\in[0,1)$. Then
$$ \operatorname{Re}\int_M^{w_0}\tilde{g}(\eta) \mathrm{d}\eta= \operatorname{Re}\left(\int_M^{M+k}+\int_{M+k}^{x_0}+\int_{x_0}^{w_0}\right)\tilde{g}(\eta) \mathrm{d}\eta.$$
Combining \eqref{1} - \eqref{3}, we have
$$\left| \operatorname{Re}\int_M^{w_0}\tilde{g}(\eta) \mathrm{d}\eta\right|<3C+(k+\frac{\pi}{\alpha_0})\epsilon.$$
This gives \eqref{0}.
\end{proof}

\begin{proposition}\label{4zyg}
Suppose $\gamma_j\ne-\infty$ for some $j\geq 1$.  Then for any $\epsilon>0$, there exist a neighborhood $U$ of $\zeta_j$ and constants $C,c>0$ such that
\begin{align}\label{4zhsz}
c\left| \mathrm{e}^{h(z)}\right|^{ \operatorname{Re}\beta_j-\epsilon}<|v(z)|<C\left| \mathrm{e}^{h(z)}\right|^{ \operatorname{Re}\beta_j+\epsilon}
\end{align}
for all $z\in U\cap \mathbb{D}$.
\end{proposition}

\begin{proof}
Choos $M > 0$ large enough so that $\{w \in \Omega :  \operatorname{Re}w\leq -M\}$ has $m$ connected components $\{\Gamma_j\}_{j=1}^m$, with each $\zeta_j$ lying on the boundary of $h^{-1}(\Gamma_j)$. We shall show that for sufficiently large $M$ the estimate \eqref{4zhsz} holds on $h^{-1}(\Gamma_j)$ whenever $\beta_j\ne-\infty$.

According to Proposition 13.4.2 in \cite{cs}, $(\varphi_t)$ acts as a group on each petal $\Delta_j$, and then $(u_t\circ \varphi_t^{-1})_{t\geq0}$ is a semicocycle for $(\varphi_t^{-1})_{t\geq 0}$ on $\Delta_j$. So if $\beta_j\ne-\infty$, by Proposition \ref{4zc} we have \eqref{4zhsz} holds for all $z \in h^{-1}(\Gamma_j) \cap\Delta_j$ provided $M$ is large enough. This settles the case when both $Q_j(x)$ and $q_j(x)$ are constants near $-\infty$.

Otherwise, if $Q_j(x)$ is strictly increasing on $x \in(-\infty, M]$, then by taking $M$ large enough we have
$$\mathrm{e}^{\beta_j-\epsilon}<|u_1(z)|<\mathrm{e}^{\beta_j+\epsilon}$$ 
for $z\in h^{-1}(\Gamma_j)$. Now let $S$ be the preimage under $h$ of the rectangle $\lfloor M,M +1\rceil\cap i\lfloor Q_j(-\infty), Q_j(M -1)\rceil$. Note that $S\subset\mathbb{D}$. Denote by $C$ and $c$ be the maximum and minimum of $|v|$ on $S$ respectively. Then for any $z\in h^{-1}(\Gamma_j)$ with $ \operatorname{Im}h(z)\geq  Q_j(-\infty)$, we can write $\operatorname{Re}h(z) = M +r-k$ where $r \in [0,1)$ and $k$ is a nonnegative integer. Then
$$\mathrm{e}^{( \operatorname{Re}\beta_j-\epsilon)k} <\left|\frac{v(\varphi_k(z))}{v(z)}\right| < \mathrm{e}^{( \operatorname{Re}\beta_j+\epsilon)k}.$$
Notice that $\varphi_k(z)\in S$. Therefore \eqref{4zhsz} holds for every $z \in h^{-1}(\Gamma_j)$ with $ \operatorname{Im}h(z) \geq Q_j(-\infty)$. The same for $z \in h^{-1}(\Gamma_j)$ that $ \operatorname{Im}h(z) \leq q_j(-\infty)$ when $q_j(x)$ is strictly decreasing.
\end{proof}

\section{Spectral radius}\label{sec3}

This short section is a computation of the spectral radii of the operators in the semigroup $(u_tC_{\varphi_t})$, and this gives an estimate for the spectral bound of its generator $A$.

\begin{theorem}\label{rad}
Let $(u_tC_{\varphi_t})$ be a semigroup of bounded operators on $A^p$, where $(\varphi_t)$ is a hyperbolic semiflow. Suppose that $(\varphi_t)$ is analytic at its self-contact points, and $(u_t)$ is continuous at these points. Then for any $t\geq0$ the spectral radius of $u_tC_{\varphi_t}$ satisfies
$$r(u_tC_{\varphi_t})\leq\max\bigl\{\mathrm{e}^{\gamma_j t} : j\geq0\bigr\}.$$
\end{theorem}

\begin{proof}
Denote the self-contact points of $(\varphi_t)$ by $\{\zeta_j\}_{j=0}^m$.

First note that the operators 
$$W_n=(\varphi_n')^{\frac{2}{p}}C_{\varphi_n}$$ 
are of norms at most $1$ on $A^p$ for any integer $n\geq0$. Define
$$\nu_t=\frac{u_t}{ (\varphi_t')^{\frac{2}{p}}}.$$
Then $(\nu_t)$ is a semicocycle for $(\varphi_t)$ and satisfies $|\nu_t(\zeta_j)|= \mathrm{e}^{\gamma_j t}$. So given any $\epsilon>0$, we can take positive integer $N$ such that 
$$|\nu_1(z)|<\max_{j\geq0} \mathrm{e}^{\gamma_j}+\epsilon$$ 
whenever $|\operatorname{Re}h(z)|\geq N/2$.

Now for any fixed $n>N$, partition $\mathbb{D}$ into at most $n-N+1$ parts $\{D_k\}$ as 
$$D_0=\{z\in\mathbb{D} :  \operatorname{Re}h(z)\leq N/2+1\},$$
$$D_{n-N}=\{z\in\mathbb{D} :  \operatorname{Re}h(z)>n- N/2\},$$
and
$$D_k=\{z\in\mathbb{D} : N/2+k< \operatorname{Re}h(z)\leq N/2+k+1\}$$
for $1\leq k\leq n-N-1$. Then for any $f\in A^p$,
\begin{align*}
\|u_nC_{\varphi_n}f\|_p^p&=\int_{\mathbb{D}}|u_n(z)|^p  |f\circ\varphi_n(z)|^p \mathrm{d}V(z)\\
&=\int_{\mathbb{D}}|\nu_n(z)|^p  |f\circ\varphi_n(z)|^p  |\varphi_n'(z)|^2 \mathrm{d}V(z)\\
&=\int_{\varphi_n(\mathbb{D})}|\nu_n\circ\varphi_n^{-1}(z)|^p  |f(z)|^p \mathrm{d}V(z)\\
&=\sum_{k=0}^{n-N}\int_{D_k\cap\varphi_n(\mathbb{D})}|\nu_n\circ\varphi_n^{-1}(z)|^p  |f(z)|^p \mathrm{d}V(z).
\end{align*}
Note that if $z\in D_k\cap\varphi_n(\mathbb{D})$, then $|\operatorname{Re}h(\varphi_j^{-1}(z))|\geq N/2$ for all indices $j$ with $1\leq j\leq k$ or $N+k+1\leq j\leq n$. Consequently,
$$|\nu_n\circ\varphi_n^{-1}(z)|<\left(\max_{j\geq0} \mathrm{e}^{\gamma_j}+\epsilon\right)^{n-N} |\nu_N\circ\varphi^{-1}_{N+k}(z)|.$$
Therefore, 
\begin{align*}
&\int_{D_k\cap\varphi_n(\mathbb{D})}|\nu_n\circ\varphi_n^{-1}(z)|^p  |f(z)|^p \mathrm{d}V(z)\\
\leq&\left(\max_{j\geq0} \mathrm{e}^{\gamma_j}+\epsilon\right)^{p(n-N)} \int_{\varphi_n(\mathbb{D})}|\nu_N\circ\varphi_{N+k}^{-1}(z)|^p  |f(z)|^p \mathrm{d}V(z)\\
\leq&\left(\max_{j\geq0} \mathrm{e}^{\gamma_j}+\epsilon\right)^{p(n-N)} \int_{\mathbb{D}}|\nu_N(z)|^p  |f\circ\varphi_{N+k}(z)|^p  |\varphi'_{N+k}(z)|^2 \mathrm{d}V(z)\\
=&\left(\max_{j\geq0} \mathrm{e}^{\gamma_j}+\epsilon\right)^{p(n-N)} \int_{\mathbb{D}}|u_N(z)|^p  |f\circ\varphi_{N+k}(z)|^p  |\varphi'_k\circ\varphi_N(z)|^2 \mathrm{d}V(z)\\
=&\left(\max_{j\geq0} \mathrm{e}^{\gamma_j}+\epsilon\right)^{p(n-N)} \left\|u_NC_{\varphi_N}(W_kf)\right\|^p\\
\leq&\left(\max_{j\geq0} \mathrm{e}^{\gamma_j}+\epsilon\right)^{p(n-N)} \|u_NC_{\varphi_N}\|^p\|f\|_p^p.
\end{align*}
Summing over $k$ we obtain
$$\|u_nC_{\varphi_n}\|\leq n^{\frac{1}{p}} \left(\max_{j\geq0} \mathrm{e}^{\gamma_j}+\epsilon\right)^{n-N} \|u_NC_{\varphi_N}\|.$$
So by the arbitrariness of $\epsilon>0$, 
$$r(u_1C_{\varphi_1})=\lim_{n\to\infty}\|u_nC_{\varphi_n}\|^{1/n}\leq\max\bigl\{\mathrm{e}^{\gamma_j t} : j\geq0\bigr\}.$$

Finally, for any $t>0$, by consider the semigroup $(u_{st}C_{\varphi_{st}})_{s\geq0}$, we can know that the spectral radius of $u_tC_{\varphi_t}$ is less than or equals to  $$\max\bigl\{\mathrm{e}^{\gamma_j t} : j\geq0\bigr\}.$$
\end{proof}

Combining Theorem \ref{rad} with \eqref{SMT0}, we have the following corollary.

\begin{corollary}\label{spectralbound}
Let $(u_tC_{\varphi_t})$ be a $C_0$-semigroup on $A^p$, where $(\varphi_t)$ is a hyperbolic semiflow. Suppose $(\varphi_t)$ is analytic at its self-contact points, and $(u_t)$ is continuous at these points. Then the spectrum of its generator $A$ satisfies
$$\sigma(A)\subset\left\lfloor-\infty, \max_{j\geq 0}\gamma_j\right\rceil.$$
\end{corollary}

\begin{remark}
Theorem \ref{main} actually shows that  $$r(u_tC_{\varphi_t})=\max\bigl\{\mathrm{e}^{\gamma_j t} : j\geq0\bigr\}.$$
\end{remark}

\section{Eigenvalues of generators and semigroups}\label{sec5}

Recall that $G$ and $g$ denote the  infinitesimal generator of $(\varphi_t)$ and $(u_t)$ respectively. A direct computation shows that the generator of $(u_tC_{\varphi_t})$ is given by
$$Af=G f'+gf=\frac{f'}{h'}+\frac{v'f}{vh'}$$
for $f\in D(A)$, where
$$D(A)=\{f\in A^p : G f'+gf\in A^p\}$$
is the domain of $A$. Consequently, for any $\lambda\in\mathbb{C}$, solving the equation $(\lambda-A)F=f$ yields
\begin{align}\label{difsolu}
F(z)=\frac{ \mathrm{e}^{\lambda h(z)}}{v(z)}\left(K-\int_0^z \mathrm{e}^{-\lambda h(\xi)}h'(\xi)v(\xi)f(\xi) \mathrm{d}\xi\right),
\end{align}
with an arbitrary constant $K\in\mathbb{C}$. We omit the detailed calculation here as it is quite straightforward.

The following well-known estimate for the boundary growth of functions in Bergman space will be useful:
\begin{align}\label{bergman}
|f(z)|\leq(1 - |z|^2)^{-\frac{2}{p}} \|f\|_p
\end{align}
for any $f\in A^p$. This estimate can be found in many references. See, for example, Section 2.4 of \cite{bergman}.

The next theorem gives characterization of the eigenvalues of the generator $A$.

\begin{theorem}\label{eigenvalue}
Let $(u_tC_{\varphi_t})$ be a $C_0$-semigroup on $A^p$, where $(\varphi_t)$ is a hyperbolic semiflow. Suppose $(\varphi_t)$ is analytic at its self-contact points, and $(u_t)$ is continuous at these points.
\begin{enumerate}
\item If $\gamma_1<\gamma_0$, then 
$$\overline{\sigma_p(A)}=\lfloor\gamma_1,\gamma_0\rceil,$$
and each interior point of this strip  (or half-plane) is an eigenvalue of $A$.

\item  If $\gamma_1>\gamma_0$, then $\sigma_p(A)=\emptyset$.
\end{enumerate}
\end{theorem}

\begin{proof}

From \eqref{difsolu} we see that $\lambda$ is an eigenvalue of $A$ if and only if $\frac{ \mathrm{e}^{\lambda h(z)}}{v(z)}\in A^p$.

First assume that $\beta_0=-\infty$. We shall show that in this case the point spectrum of $A$ is empty. In fact, for any $L>0$, there exists $M>0$ such that $|u_1(z)|<\mathrm{e}^{-L}$ whenever $\operatorname{Re}h(z)\geq M$. Fix any point $z_0\in\mathbb{D}$ with $\operatorname{Re}h(z_0)\geq M$, and let 
$$C=\max\{|v(\varphi_t(z_0))| : t\in[0,1]\}.$$ 
For $t$ sufficiently large, write $t=k+r$ with a positive integer $k$ and $r\in[0,1)$, then
\begin{align}\label{daibu}
v(\varphi_t(z_0))< C\mathrm{e}^{-Lk}\leq C\mathrm{e}^{-L(t-1)}.
\end{align}
Since $\varphi_t(z_0)$ approaches $\zeta_0$ nontangentially as $t\to\infty$ (see Proposition 8.7.3 in \cite{cs}), we have
$$\lim_{t\to+\infty}(1-|\varphi_t(z_0)|)^{1/t}=\mathrm{e}^{-\alpha_0}.$$ 
Consequently, for any $\lambda\in\mathbb{C}$, we have 
\begin{align*}
\left|\frac{ \mathrm{e}^{\lambda h(\varphi_t(z_0))}}{v(\varphi_t(z_0))}\right|>&C' \mathrm{e}^{(\operatorname{Re}\lambda+L)t}\\
\geq&C'\left(\frac{1}{1-|\varphi_t(z_0)|}\right)^{\frac{\operatorname{Re}\lambda+L}{2\alpha_0}}
\end{align*}
when $t$ is large enough, where $C'$ is a constant independent of $t$. This contradicts the growth estimate \eqref{bergman} for functions in $A^p$ since $L$ can be taken arbitrarily large. Therefore, $\frac{ \mathrm{e}^{\lambda h(z)}}{v(z)}\notin A^p$ for any $\lambda\in\mathbb{C}$, and hence we have $\sigma(A)=\emptyset$.

Now suppose that $\beta_0\ne-\infty$. By Proposition \ref{4zc}, for any $M$ large enough, the function $v(z)$ is bounded below from zero on $\{z : M \leq  \operatorname{Re}h(z) \leq3M\}$. Therefore,
$$\frac{ \mathrm{e}^{\lambda h(z)}}{v\circ\varphi_{2M}(z)}$$
is bounded on $\{z : -M \leq \operatorname{Re}h(z)\leq M\}$. Then
$$\left|\frac{\mathrm{e}^{\lambda h(z)}}{v(z)}\right|^p=\left|\frac{ \mathrm{e}^{\lambda h(z)}}{v\circ\varphi_{2M}(z)}\right|^p\cdot|u_{2M}(z)|^p,$$
which is integrable on $\{z : -M \leq\operatorname{Re}h(z) \leq M\}$. Hence $\lambda$ is an eigenvalue of $A$ if and only if $\frac{ \mathrm{e}^{\lambda h(z)}}{v(z)}\in A^p_{\zeta_j}$ at every self-contact point $\zeta_j$ for $j\geq 0$.

By Corollary \ref{zhishu} and Proposition \ref{4zc}, $\frac{ \mathrm{e}^{\lambda h(z)}}{v(z)}\in A^p_{\zeta_0}$ if $ \operatorname{Re}\lambda<\gamma_0$, and $\frac{ \mathrm{e}^{\lambda h(z)}}{v(z)}\notin A^p_{\zeta_0}$ if $ \operatorname{Re}\lambda>\gamma_0$. For $j\geq 1$, if $\beta_j\ne-\infty$, Corollary \ref{zhishu} together with Proposition \ref{4zyg} also yields $\frac{ \mathrm{e}^{\lambda h(z)}}{v(z)}\in A^p_{\zeta_j}$ if $ \operatorname{Re}\lambda>\gamma_j$, and $\frac{ \mathrm{e}^{\lambda h(z)}}{v(z)}\notin A^p_{\zeta_j}$ if $ \operatorname{Re}\lambda<\gamma_j$. 

Thus, it only remains to consider the situation where $\beta_j=-\infty$ at some self-contact point $\zeta_j$ with $j\geq 1$. If we can show in this situation that $\frac{ \mathrm{e}^{\lambda h(z)}}{v(z)}\in A^p_{\zeta_j}$ for all $\lambda\in\mathbb{C}$, then the proof of the theorem will be completed. To verify this fact, fix $L>\alpha_j/p-\operatorname{Re}\lambda$, and take $M>0$ large enough so that $u_1(z)< \mathrm{e}^{-L}$ whenever $ \operatorname{Re}h(z)\leq-M$. Morever, since $\varphi_1$ is analytic at $\zeta_j$, we may assume that
$$|\varphi_1'(z)|> \mathrm{e}^{-\alpha_j/2}$$
when $ \operatorname{Re}h(z)\leq-M$. Define 
$$U_k=h^{-1}\left(\Omega\cap\lfloor-M-k-1,-M-k\rceil\right)$$
for $k\geq0$. Then $\left|\frac{ \mathrm{e}^{\lambda h(z)}}{v(z)}\right|^p$ is integrable on $U_0$. By noting that $\varphi_1(U_{k+1})\subset U_k$ for any $k\geq 0$, we have
\begin{align*}
\int_{U_{k+1}}\left|\frac{ \mathrm{e}^{\lambda h(z)}}{v(z)}\right|^p \mathrm{d}V(z)\leq \mathrm{e}^{-pL-p\mathrm{\lambda}+\alpha_j} \int_{U_k}\left|\frac{ \mathrm{e}^{\lambda h(z)}}{v(z)}\right|^p \mathrm{d}V(z).
\end{align*}
Consequently, 
$$\int_{\cup U_k}\left|\frac{ \mathrm{e}^{\lambda h(z)}}{v(z)}\right|^p \mathrm{d}V(z)=\sum_{k=0}^\infty\int_{U_k}\left|\frac{ \mathrm{e}^{\lambda h(z)}}{v(z)}\right|^p \mathrm{d}V(z)<\infty,$$
which means that $\frac{ \mathrm{e}^{\lambda h(z)}}{v(z)}\in A^p_{\zeta_j}$. This completes the proof.
\end{proof}

As a corollary, we can immediately get a description of the point spectra of the operators $u_tC_{\varphi_t}$ as follows.

\begin{remark}\label{2yue25}
By taking $v(z)\equiv1$ we obtain the composition $C_0$-semigroup $(C_{\varphi_t})$. Then $\lambda\in\mathbb{C}$ is an eigenvalue of its generator $A$ if and only if $\mathrm{e}^{\lambda h(z)}\in A^p$. Therefore, by Corollary \ref{zhishu}, we have
$$\sigma_p(A)=\{\lambda : \alpha_1<\operatorname{Re}\lambda<\alpha_0\},$$
where $\alpha_0$ is the spectral value of $(\varphi_t)$ at its Denjoy-Wolff point $\zeta_0$, and $\gamma_1$ is its maximal spectral value at the repelling fixed points. If $\zeta_0$ is the only fixed point of $(\varphi_t)$, we simply set $\gamma_1=-\infty.$ Note that this is a spectial case of a more general theorem due to \cite{tezheng2}, see Theorem 3 therein, which characterizes the eigenvalues of the generator of $(C_{\varphi_t})$ without requiring $(\varphi_t)$ to be analytic at its self-contact points.
\end{remark}

\section{Proof of Theorem \ref{main} and Theorem \ref{ess}}\label{sec6}

This entire section constitutes complete proofs of Theorem \ref{main} and Theorem \ref{ess}.  With Theorem \ref{eigenvalue} in hand, the key issue we need to address is when $\lambda-A$ fails to be surjective. Given equation \eqref{difsolu}, for any  $f\in A^p$ and $\lambda\in\mathbb{C}$, we set the $(1,0)$-form
$$\omega_{\lambda,f}= \mathrm{e}^{-\lambda h(\xi)}h'(\xi)v(\xi)f(\xi) \mathrm{d}\xi$$
on $\mathbb{D}$. Then, by \eqref{difsolu}, the operator $\lambda-A$ is surjective if and only if for any  $f\in A^p$, there exists a constant $K=K(f)$ such that
$$\frac{ \mathrm{e}^{\lambda h(z)}}{v(z)}\left(K-\int_0^z\omega_{\lambda,f}\right)\in A^p.$$
Therefore, we need to establish several estimates for integrals of $\omega_{\lambda,f}$.

We begin by considering the integral $\int_z^{\zeta_0}\omega_{\lambda,f}$, where the integration path is taken along the curve $t\mapsto\varphi_t(z)$ for $t\in[0,+\infty)$.

\begin{lemma}\label{jifenshoulian}
If $ \operatorname{Re}\lambda>\gamma_0$, then for each $f\in A^p$, the integral $\int_z^{\zeta_0}\omega_{\lambda,f}$ gives an analytic function with respect to $z\in\mathbb{D}$.
\end{lemma}

\begin{proof}
We first assume that $\beta_0\ne-\infty$. The proof then proceeds in two steps.

For the first step, we need to verify that the integral converges for each fixed point $z\in \mathbb{D}$. Note that along the curve $t\mapsto\varphi_t(z)$, we have
$$\omega_{\lambda,f}= \mathrm{e}^{-\lambda h(z) - \lambda t}v\bigl(\varphi_t(z)\bigr) f\bigl(\varphi_t(z)\bigr) \mathrm{d}t.$$
Thus, it suffices to check the convergence of
$$\int_0^{+\infty} \mathrm{e}^{-\lambda t}v\bigl(\varphi_t(z)\bigr) f\bigl( \varphi_t(z)\bigr) \mathrm{d}t$$ 
for each fixed $z\in\mathbb{D}$. 

Because $\varphi_t(z)$ approaches $\zeta_0$ non-tangentially as $t\to\infty$ (see Proposition 8.7.3 in \cite{cs}), we have
$$\lim_{t\to +\infty}(1-|\varphi_t(z)|)^{1/t}= \mathrm{e}^{-\alpha_0}.$$ 
Hence by the growth estimate \eqref{bergman}, for any $\epsilon>0$ we have
\begin{align}\label{xiangtong1}
|f\bigl(\varphi_t(z)\bigr)|<\left(\frac{1}{1-|\varphi_t(z)|}\right)^{2/p}\cdot\|f\|_p< \mathrm{e}^{\left(\frac{2\alpha_0}{p}+\epsilon\right)t}\cdot\|f\|_p
\end{align}
when $t$ is large enough. Meanwhile, by Proposition \ref{4zc}, for sufficiently large $t$ we have
$$|v\bigl(\varphi_t(z)\bigr)|<C \mathrm{e}^{( \operatorname{Re}\beta_0+\epsilon)t},$$
where the constant $C>0$ is independent of $t$. Therefore, 
$$|\mathrm{e}^{-\lambda t}v\bigl(\varphi_t(z)\bigr) f\bigl( \varphi_t(z)\bigr)|<C \mathrm{e}^{(-\operatorname{Re}\lambda+\gamma_0+2\epsilon)t}\cdot\|f\|_p,$$
which is integrable on $[0,+\infty)$ when $\epsilon$ is  taken sufficiently small. This shows that $\int_z^{\zeta_0}\omega_{\lambda,f}$ is well defined for every $z\in\mathbb{D}$.

Now we will show that $\int_z^{\zeta_0}\omega_{\lambda,f}$ is analytic.  We shall work on $\Omega$ for this purpose. Then it  is sufficient to prove that for any two numbers $y_1<y_2$ in $(q_0(+\infty), Q_0(+\infty))$,
$$\lim_{x\to\infty}\int_{x+iy_1}^{x+iy_2}(h^{-1})^*\omega_{\lambda,f}=0,$$
where $(h^{-1})^*\omega_{\lambda,f}$ denote the pull-back of $\omega_{\lambda,f}$ under $h^{-1}$. A direct calculation gives
$$(h^{-1})^*\omega_{\lambda,f}= \mathrm{e}^{-\lambda\eta}\cdot v\circ h^{-1}(\eta)f\circ h^{-1}(\eta) \mathrm{d}\eta.$$
Since the strip $i\lfloor y_1, y_2\rceil$ is contained in a nontangential region at $\zeta_0$, the same argument as for \eqref{xiangtong1} gives
$$|f\circ h^{-1}(x+iy)|<\mathrm{e}^{\left(\frac{2\alpha_0}{p}+\epsilon\right)x}\cdot\|f\|_p$$
for all $y\in[y_1, y_2]$ when $x$ is sufficiently large. Applying Proposition \ref{4zc} again, we obtain
\begin{align*}
\left|\int_{x+iy_1}^{x+iy_2} (h^{-1})^*\omega_{\lambda,f}\right|\leq C  \mathrm{e}^{(-\operatorname{Re}\lambda+\gamma_0+2\epsilon)x}\cdot|y_2-y_1|,
\end{align*}
which tends to $0$ as $x\to+\infty$. Therefore, $\int_z^{\zeta_0}\omega_{\lambda,f}$ is analytic on $\mathbb{D}$.

If $\beta_0=-\infty$, the proof is essentially the same. The only difference is that instead of using Proposition \ref{4zc}, one should follow the discussion concerning \eqref{daibu} in Theorem \ref{eigenvalue} to see that for any nontangential region $S$ at $\zeta_0$ and any $L>0$,
$$|v(z)|<C \mathrm{e}^{-L\operatorname{Re}h(z)}$$
for all $z\in S$ sufficiently close to $\zeta_0$. With this bound a similar computation shows that the integral converges and defines an analytic function.
\end{proof}

Analogously, for eachself-contact points $\zeta_j$ of $(\varphi_t)$ other than the Denjoy-Wolff point $\zeta_0$, we can consider the integral $\int_{\zeta_j}^z\omega_{\lambda,f}$ along the backward orbit $t\mapsto \varphi_{t}(z)$ for $t\in(-\infty,0]$. Formally, this integral is defined only on the petal $\Delta_j$. One can show that it defines an analytic function on $\Delta_j$ when $\operatorname{Re}\lambda<\gamma_j$. Indeed, we may assume that $\gamma_j\ne-\infty$. Let $(\tilde{\varphi}_t)=(\varphi_t^{-1})$, which is a flow on $\Delta_j$ with Denjoy-Wolff point $\tilde{\zeta}_0=\zeta_j$. Define 
$$\tilde{u}_t=\frac{1}{u_t\circ\varphi_t^{-1}}=\frac{v\circ\tilde{\varphi}_t}{v}$$
for $t\geq 0$ on $\Delta_j$. Then $(\tilde{u}_t)$ is a semicocycle for $(\tilde{\varphi}_t)_{t\geq0}$ on $\Delta_j$. Observe that $\tilde{\gamma}_0=-\gamma_j$, hence $\operatorname{Re}(-\lambda)>\tilde{\gamma}_0$. So the analyticity of $\int_{\zeta_j}^z\omega_{\lambda,f}$ on $\Delta_j$ now follows from Lemma \ref{jifenshoulian} once we notice that
\begin{align*}
\int_{\zeta_j}^z\omega_{\lambda,f}=&\mathrm{e}^{-\lambda h(z) }\int_{-\infty}^0\mathrm{e}^{-\lambda t}v\left(\varphi_t(z)\right)f\left(\varphi_t(z)\right)\mathrm{d}t\\
=&-\mathrm{e}^{-\lambda h(z)}\int_0^{+\infty}\mathrm{e}^{\lambda t}v(\tilde{\varphi}_t(z))f(\tilde{\varphi}_t(z))\mathrm{d}t.
\end{align*}
Furthermore, when $\operatorname{Re}\lambda<\gamma_j$, we can actually extend $\int_{\zeta_j}^z\omega_{\lambda,f}$ to the whole disk $\mathbb{D}$ by setting
$$\int_{\zeta_j}^z\omega_{\lambda,f}=\int_{\zeta_j}^{z'}\omega_{\lambda,f}+\int_{z'}^z\omega_{\lambda,f}$$
for some $z'\in\Delta_j$. This continuation is independent of the choice of $z'$. Thus we have prove the following lemma.

\begin{lemma}\label{zaibianyige}
If $\operatorname{Re}\lambda<\gamma_j$ for some $j\geq 1$, then $\int^z_{\zeta_j}\omega_{\lambda,f}$ gives an analytic function with respect to $z\in\mathbb{D}$ for each $f\in A^p$.
\end{lemma}

\begin{remark}\label{zuihouremark}
The proofs of Lemma \ref{jifenshoulian} and \ref{zaibianyige} shows that when considering the integrals  $\int^z_{\zeta_j}\omega_{\lambda,f}$ for $j\geq0$, the integration path can be taken as any smooth curve from $z$ to $\zeta_j$ that approaches $\zeta_j$ nontangentially.
\end{remark}

In what follows, we will proceed case by case to prove five claims, all aimed at searching a constant $K$ such that the function $F(z)$ in \eqref{difsolu} belongs to the space $A^p$ for different $\lambda\in\mathbb{C}$.

\begin{claim}\label{>0}
If $ \operatorname{Re}\lambda>\gamma_0$, then for any $f\in A^p$,
$$\frac{ \mathrm{e}^{\lambda h(z)}}{v(z)}\left(K-\int_0^z\omega_{\lambda,f}\right)\in A^p_{\zeta_0}$$
if and only if $K=\int_0^{\zeta_0}\omega_{\lambda,f}$.
\end{claim}

\begin{proof}
Let 
$$\tilde{v}(z)=\frac{v(z)}{\Pi_{j\geq1}(z-\zeta_j)^{c_j}},$$
where $\{c_j\}_{j\geq 1}$ are constants to be determined, and we set $c_j=0$ whenever $\gamma_j=-\infty$. Define
$$\tilde{u}_t(z)=\frac{\tilde{v}(\varphi_t(z))}{\tilde{v}(z)},$$
which gives the semicocycle for $(\varphi_t)$ with semicoboundary $\tilde{v}$. Note that  $\tilde{\beta}_0=\beta_0$ and $\tilde{\beta}_j=\beta_j+c_j\alpha_j$, hence by taking $\{c_j\}$ appropriately we may have
$$\max_{j\geq0}\tilde{\gamma}_j<\operatorname{Re}\lambda.$$ 
Denote
$$\tilde{\omega}_{\lambda,f}=\mathrm{e}^{-\lambda h(\xi)}h'(\xi)\tilde{v}(\xi)f(\xi) \mathrm{d}\xi.$$
Then 
$$\left|\frac{ \mathrm{e}^{\lambda h(z)}}{\tilde{v}(z)}\right|\int_z^{\zeta_0}|\tilde{\omega}_{\lambda,f}|=\int_0^{+\infty} \mathrm{e}^{-\operatorname{Re}\lambda t}\left|\tilde{u}_tC_{\varphi_t}f(z)\right| \mathrm{d}t.$$
By Theorem \ref{rad}, $\lim_{t\to+\infty}\|\tilde{u}_tC_{\varphi_t}\|^{1/t}< \mathrm{e}^{\operatorname{Re}\lambda}$. Therefore, we have
$$\left|\frac{ \mathrm{e}^{\lambda h(z)}}{\tilde{v}(z)}\right|\int_z^{\zeta_0}|\tilde{\omega}_{\lambda,f}|\in L^p(\mathbb{D}).$$
However, as $\frac{ \mathrm{e}^{\lambda h(z)}}{v(z)}\int_z^{\zeta_0}\omega_{\lambda,f}$ is dominated by a constant multiple of $\frac{ \mathrm{e}^{\lambda h(z)}}{\tilde{v}(z)}\int_z^{\zeta_0}|\tilde{\omega}_{\lambda,f}|$ near $\zeta_0$, the former belongs to $A^p_{\zeta_0}$. That is, 
$$\frac{ \mathrm{e}^{\lambda h(z)}}{v(z)}\left(K-\int_0^z\omega_{\lambda,f}\right)\in A^p_{\zeta_0}$$
when $K=\int_0^{\zeta_0}\omega_{\lambda,f}$. 

Finally, the proof of Theorem \ref{eigenvalue} tells us that $\frac{ \mathrm{e}^{\lambda h(z)}}{v(z)}\notin A_{\zeta_0}^p$. Therefore, if $K\ne\int_0^{\zeta_0}\omega_{\lambda,f}$, the function $\frac{ \mathrm{e}^{\lambda h(z)}}{v(z)}\left(K-\int_0^z\omega_{\lambda,f}\right)$ cannot belong to $A^p_{\zeta_0}$.
\end{proof}

The proof of Claim \ref{>0} actually gives the following claim. 

\begin{claim}\label{others}
Let $\zeta$ be a point on $\partial \mathbb{D}$ that is not a fixed point of $(\varphi_t)$. Then
$$\frac{ \mathrm{e}^{\lambda h(z)}}{v(z)}\left(K-\int_0^{z}\omega_{\lambda,f}\right)\in A^p_{\zeta}$$
for every $f\in A^p$ and $\lambda, K\in\mathbb{C}$. 
\end{claim}

\begin{proof}
Given $\lambda\in\mathbb{C}$. Define
$$\tilde{v}(z)=\frac{v(z)}{\Pi_{j\geq0}(z-\zeta_j)^{c_j}},$$
and let $\tilde{u}_t=\frac{\tilde{v}\circ\varphi_t}{\tilde{v}},$
which gives the semicocycle for $(\varphi_t)$ with semicoboundary $\tilde{v}$.
Choose $c_k$ so that 
$$\operatorname{Re}\lambda>\max_{j\geq0}\tilde{\gamma}_j.$$
The proof of Claim \ref{>0} then shows that 
$$\left|\frac{ \mathrm{e}^{\lambda h(z)}}{\tilde{v}(z)}\right|\int_z^{\zeta_0}|\tilde{\omega}_{\lambda,f}|\in L^p(\mathbb{D}).$$
where
$$\tilde{\omega}_{\lambda,f}=\mathrm{e}^{-\lambda h(\xi)}h'(\xi)\tilde{v}(\xi)f(\xi) \mathrm{d}\xi.$$
Note that near any non-fixed boundary point $\zeta$ of  $(\varphi_t)$, the function $\frac{ \mathrm{e}^{\lambda h(z)}}{v(z)}\int_z^{\zeta_0}\omega_{\lambda,f}$ is dominated by a constant multiple of $\frac{ \mathrm{e}^{\lambda h(z)}}{\tilde{v}(z)}\int_z^{\zeta_0}|\tilde{\omega}_{\lambda,f}|$. So we have
$$\frac{ \mathrm{e}^{\lambda h(z)}}{v(z)}\int_z^{\zeta_0}\omega_{\lambda,f}\in A^p_\zeta.$$
Moreover, the proof of Theorem \ref{eigenvalue} shows that $\frac{ \mathrm{e}^{\lambda h(z)}}{v(z)}\in A^p_\zeta.$ Therefore, 
 $$\frac{ \mathrm{e}^{\lambda h(z)}}{v(z)}\left(K-\int_0^{z}\omega_{\lambda,f}\right)\in A^p_{\zeta}$$
for any $K\in\mathbb{C}$. 
\end{proof}

\begin{claim}\label{>j}
If $ \operatorname{Re}\lambda>\gamma_j$ for some $j\geq 1$, then
$$\frac{ \mathrm{e}^{\lambda h(z)}}{v(z)}\left(K-\int_0^{z}\omega_{\lambda,f}\right)\in A^p_{\zeta_j}$$
for any $f\in A^p$ and $K\in\mathbb{C}$. 
\end{claim}

\begin{proof}
Define 
$$\tilde{v}(z)=\frac{v(z)}{\Pi_{k\ne j}(z-\zeta_k)^{c_k}},$$ 
and let $\tilde{u}_t=\frac{\tilde{v}\circ\varphi_t}{\tilde{v}}.$ Choose the constants $\{c_k\}$ so that 
$$\operatorname{Re}\lambda>\gamma_j=\max_{k\geq0}\tilde{\gamma}_k.$$ 
Then the proof of Claim \ref{>0} shows that 
$$\frac{ \mathrm{e}^{\lambda h(z)}}{\tilde{v}(z)}\int_z^{\zeta_0}|\tilde{\omega}_{\lambda,f}|\in L^p(\mathbb{D}),$$
where
$$\tilde{\omega}_{\lambda,f}=\mathrm{e}^{-\lambda h(\xi)}h'(\xi)\tilde{v}(\xi)f(\xi) \mathrm{d}\xi.$$
Moreover, by Theorem \ref{eigenvalue}, 
$$\frac{ \mathrm{e}^{\lambda h(z)}}{\tilde{v}(z)}\in A^p_{\zeta_j}.$$
Consequently,
$$\frac{ \mathrm{e}^{\lambda h(z)}}{\tilde{v}(z)}\int_0^z|\tilde{\omega}_{\lambda,f}|=\frac{ \mathrm{e}^{\lambda h(z)}}{\tilde{v}(z)}\int_0^{\zeta_0}|\tilde{\omega}_{\lambda,f}|-\frac{ \mathrm{e}^{\lambda h(z)}}{\tilde{v}(z)}\int_z^{\zeta_0}|\tilde{\omega}_{\lambda,f}|\in L^p_{\zeta_j}(\mathbb{D}).$$ 
Since $\frac{ \mathrm{e}^{\lambda h(z)}}{v(z)}\int_0^{z}\omega_{\lambda,f}$ is dominated by a constant multiple of $\frac{ \mathrm{e}^{\lambda h(z)}}{\tilde{v}(z)}\int_0^z|\tilde{\omega}_{\lambda,f}|$ near $\zeta_j$, it belongs to $A^p_{\zeta_j}.$ Also, the proof of Theorem \ref{eigenvalue} shows that 
$$\frac{ \mathrm{e}^{\lambda h(z)}}{v(z)}\in A^p_{\zeta_j}.$$ 
Therefore, 
$$\frac{ \mathrm{e}^{\lambda h(z)}}{v(z)}\left(K-\int_0^{z}\omega_{\lambda,f}\right)\in A^p_{\zeta_j}$$
for any $K\in\mathbb{C}$. 
\end{proof}

Before we move on to the remaining two claims, we need a following lemma.

\begin{lemma}\label{lastlemma}
If $\gamma_0\ne-\infty$, then there exist a neighborhood $U$ of $\zeta_0$ and a constant $C>0$ such that
$$ |v(z)|\leq C |v(z')|$$
for all $z, z'\in U\cap\mathbb{D}$ with $ \operatorname{Re}h(z)= \operatorname{Re}h(z').$ 

Moreover, if we also have $\gamma_j\ne -\infty$ for some $j\geq 1$, then there exist a neighborhood $V$ of $\zeta_j$ and $C'>0$ such that
\begin{align}\label{zongxiang}
 |v(z)|\leq C'|v(z')|
\end{align}
for all $z, z'\in V\cap\mathbb{D}$ with $ \operatorname{Re}h(z)= \operatorname{Re}h(z').$ 
\end{lemma}

\begin{proof}
The first  statement follows directly from inequality \eqref{3} in the proof of Proposition \ref{4zc}.

Now suppose $\gamma_j\ne -\infty$ for some $j\geq 1$. By considering $(u_t\circ \varphi_t^{-1})_{t\geq0}$ as a semicocycle for $(\varphi_t^{-1})_{t\geq 0}$ on the petal $\Delta_j$, we see that \eqref{zongxiang} holds whenever $z$ and $z'$ both lie in $V\cap\Delta_j$ for some neighborhood $V$ of $\zeta_j$. Hence the lemma is proved if both $q_j(x)$ and $Q_j(x)$ are constant near $-\infty$. 

Otherwise, assume that $Q_j(x)$ is strictly increasing near $-\infty$. Let 
$$D=\{z\in\mathbb{D} : Q_j(-\infty)<\operatorname{Im}h(z)<Q_j(-M-1)\},$$
where $M>0$ is taken large enough so that \eqref{4zhsz} holds whenever $z\in D$ and $\operatorname{Re}h(z)\leq -M$. 
Since $Q_j(x)$ is strictly increasing for $x\leq-M$, it is clear that $u_t\in(H^\infty(D))^{-1}$ for any $t\geq 0.$ So by Lemma 2.1 in [4],  both $\|u_t\|_{H^\infty(D)}$ and $\|1/u_t\|_{H^\infty(D)}$ are bounded for $t\in[0,1]$. Then, using an argument similar to the proof of \eqref{3}, one can see that \eqref{zongxiang} also holds on $V\cap D$ for some neighborhood $V$ of $\zeta_j$. The same reasoning applies when $q_j(x)$ is strictly decreasing.
\end{proof}

For the next two claims, we shall work on $\Omega=h(\mathbb{D})$ instead of the unit disk. Note that, according to Theorem \ref{ana}, a function $H(z)$ belongs to $L^p_{\zeta_j}(\mathbb{D})$ if and only if 
$$\int_{\Gamma_j}|H\circ h^{-1}(w)|^p \mathrm{e}^{-2\alpha_j  \operatorname{Re}w} \mathrm{d}V(w)<\infty,$$
where $\Gamma_j=\Gamma_j(M)$ for some $M>0$ large enough, which is defined as in Remark \ref{remarkchang}.

\begin{claim}\label{<j}
If $ \operatorname{Re}\lambda<\gamma_j$ for some $j\geq 1$, then for any $f\in A^p$,
$$\frac{ \mathrm{e}^{\lambda h(z)}}{v(z)}\left(K-\int_0^z\omega_{\lambda,f}\right)\notin A^p_{\zeta_j}$$
whenever $K\ne\int_0^{\zeta_j}\omega_{\lambda,f}$.

Moreover, if we also have $\beta_0\ne-\infty$, then 
$$\frac{ \mathrm{e}^{\lambda h(z)}}{v(z)}\int^z_{\zeta_j}\omega_{\lambda,f}\in A^p_{\zeta_j}.$$

\end{claim}

\begin{proof}
We may assume that $\gamma_j\ne-\infty.$ Let $(\tilde{\varphi}_t)=(\varphi_t^{-1})$, which is a flow on $\Delta_j$ with Denjoy-Wolff point $\tilde{\zeta}_0=\zeta_j$. Define
$$\tilde{u}_t=\frac{1}{u_t\circ\varphi_t^{-1}}=\frac{v\circ\tilde{\varphi}_t}{v}$$
for $t\geq 0$ on $\Delta_j$. Then $(\tilde{u}_t)$ is a semicocycle for $(\tilde{\varphi}_t)_{t\geq0}$ on $\Delta_j$. Note that $\tilde{\gamma}_0=-\gamma_j$, hence $\operatorname{Re}(-\lambda)>\tilde{\gamma}_0$. Because $\Delta_j$ is conformal with respect to $\mathbb{D}$ at $\zeta_j$, see Remark \ref{remarkchang}, so Claim \ref{>0} gives that
\begin{align}\label{<j1}
\frac{ \mathrm{e}^{\lambda h(z)}}{v(z)}\int^z_{\zeta_j}\omega_{\lambda,f}\in A^p_{\zeta_j}(\Delta_j).
\end{align}
On the other hand, the proof of Theorem \ref{eigenvalue} shows that $\frac{ \mathrm{e}^{\lambda h(z)}}{v(z)}\notin A^p_{\zeta_j}(\Delta_j)$. So when $K\ne\int_0^{\zeta_j}\omega_{\lambda,f}$, the function $\frac{ \mathrm{e}^{\lambda h(z)}}{v(z)}\left(K-\int_0^z\omega_{\lambda,f}\right)$ is not in $A^p_{\zeta_j}(\Delta_j)$, and therefore it does not belong to $A^p_{\zeta_j}$.

Now suppose additionally that  $\beta_0\ne-\infty$. Note that \eqref{<j1} means
$$\int_{\Gamma_j\cap\Delta_j}\left|\frac{ \mathrm{e}^{\lambda w}}{v\circ h^{-1}(w)}\int^{h^{-1}(w)}_{\zeta_j}\omega_{\lambda,f}\right|^p \mathrm{e}^{-2\alpha_j  \operatorname{Re} w} \mathrm{d}V(w)<\infty$$
for some $\Gamma_j=\Gamma_j(M)$ with $M>0$ large enough. Hence we can pick $b\in(q_j(-\infty),Q_j(-\infty))$ such that 
\begin{align}\label{<j2}
\int_{-\infty}^{-M}\left|\frac{ \mathrm{e}^{\lambda x}}{v\circ h^{-1}(x+ib)}\int^{h^{-1}(x+ib)}_{\zeta_j}\omega_{\lambda,f}\right|^p \mathrm{e}^{-2\alpha_j x} \mathrm{d}x<\infty.
\end{align}
Now, for any $z\in\Gamma_j$, write
$$\frac{ \mathrm{e}^{\lambda h(z)}}{v(z)}\int^z_{\zeta_j}\omega_{\lambda,f}=H_1(z)+H_2(z),$$
where
$$H_1(z)=\frac{ \mathrm{e}^{\lambda h(z)}}{v(z)}\int^{h^{-1}(\operatorname{Re}h(z)+ib)}_{\zeta_j}\omega_{\lambda,f}$$
and
$$H_2(z)=\frac{ \mathrm{e}^{\lambda h(z)}}{v(z)}\int^z_{h^{-1}(\operatorname{Re}h(z)+ib)}\omega_{\lambda,f}.$$
By Lemma \ref{lastlemma},
$$|H_1(z)|\leq C\left|\frac{ \mathrm{e}^{\lambda h(z)}}{v\circ h^{-1}(\operatorname{Re}h(z)+ib)}\int^{h^{-1}(\operatorname{Re}h(z)+ib)}_{\zeta_j}\omega_{\lambda,f}\right|.$$
An then \eqref{<j2} show that $H_1\in L^p_{\zeta_j}(\mathbb{D})$. Thus it remains to prove that $H_2\in L^p_{\zeta_j}(\mathbb{D})$.

To see this, for any $w=x+iy\in\Gamma_j$, by Lemma \ref{lastlemma},
\begin{align*}
|H_2\circ h^{-1}(w)|^p&\leq\int_{\min\{b,y\}}^{\max\{b,y\}}\left|\frac{v\circ h^{-1}(x+it)}{v\circ h^{-1}(x+iy)}\right|^p |f(x+it)|^p\mathrm{d}t\\
&\leq C\int_{\min\{b,y\}}^{\max\{b,y\}}|f(x+it)|^p\mathrm{d}t,
\end{align*}
Therefore,
\begin{align*}
&\int_{\Gamma_j}|H_2\circ h^{-1}(w)|^p \mathrm{e}^{-2\alpha_j  \operatorname{Re}w} \mathrm{d}V(w)\\
\leq&C\int_{-\infty}^{M}\mathrm{d}x\int_{q(x)}^{Q(x)}\mathrm{d}y\int_{\min\{b,y\}}^{\max\{b,y\}}|f(x+it)|^p\mathrm{e}^{-2\alpha_jx}\mathrm{d}t\\
\leq&C\Bigl(Q(M)-q(M)\Bigr)\cdot\int_{-\infty}^{M}\mathrm{d}x\int_{q(x)}^{Q(x)}|f(x+it)|^p\mathrm{e}^{-2\alpha_jx}\mathrm{d}t\\
\leq&C\Bigl(Q(M)-q(M)\Bigr)\cdot\|f\|_p^p<\infty.
\end{align*}
This shows that $H_2\in L^p_{\zeta_j}(\mathbb{D})$ as well.
\end{proof}

\begin{claim}\label{<0}
Suppose $\gamma_1\ne-\infty$. If $ \operatorname{Re}\lambda<\gamma_0$, then 
$$\frac{ \mathrm{e}^{\lambda h(z)}}{v(z)}\left(K-\int_0^{z}\omega_{\lambda,f}\right)\in A^p_{\zeta_0}$$
for any $f\in A^p$ and $K\in\mathbb{C}$. 
\end{claim}

\begin{proof}
Let $D$ be the interior of
$$h^{-1}\left(i\left\lfloor \frac{3q_1(-\infty)+Q_1(-\infty)}{4}, \frac{q_1(-\infty)+3Q_1(-\infty)}{4}\right\rceil\right).$$
Then $(\tilde{\varphi}_t)=(\varphi^{-1}_t)_{t\geq0}$ gives a semiflow on $D$ with spectral value $\tilde{\alpha}_1=2\alpha_1$ at its only repelling fixed point $\tilde{\zeta}_1=\zeta_0$. Define
$$\tilde{u}_t=\frac{1}{u_t\circ\varphi_t^{-1}}=\frac{v\circ\tilde{\varphi}_t}{v}.$$
Then $(\tilde{u}_t)$ is a semicocycle for the $(\tilde{\varphi}_t)$ on $D$, and 
\begin{align*}
\tilde{\gamma}_1=\frac{4}{p}\alpha_1-\operatorname{Re}\beta_0=\frac{2}{p}(\alpha_0+2\alpha_1)-\gamma_0.
\end{align*}

Given $f\in A^p$, put $ \tilde{f}=f \cdot\mathrm{e}^{-\frac{2}{p}(\alpha_0+2\alpha_1)h}$. And let $\lambda'=\lambda-\frac{2}{p}(\alpha_0+2\alpha_1)$. Then we have $\operatorname{Re}(-\lambda')>\tilde{\gamma}_1$. So by Claim \ref{>j}, we can take $b\in(q_1(-\infty),Q_1(-\infty))$ such that 
\begin{align}\label{<01}
\int_{M}^{+\infty}\left|\frac{ \mathrm{e}^{\lambda' x}}{v\circ h^{-1}(x+ib)}\int^{x+ib}_{\zeta_j}\omega_{\lambda',\tilde{f}}\right|^p \mathrm{e}^{2\alpha_1 x} \mathrm{d}x<\infty.
\end{align}

Now write 
$$\frac{ \mathrm{e}^{\lambda h(z)}}{v(z)}\int_0^z\omega_{\lambda,f}=H_1(z)+H_2(z)+\frac{ \mathrm{e}^{\lambda h(z)}}{v(z)}\int_0^{h^{-1}(ib)}\omega_{\lambda,f},$$
where
$$H_1(z)=\frac{ \mathrm{e}^{\lambda h(z)}}{v(z)}\int^{h^{-1}(\operatorname{Re}h(z)+ib)}_{h^{-1}(ib)}\omega_{\lambda,f},$$
and
$$H_2(z)=\frac{ \mathrm{e}^{\lambda h(z)}}{v(z)}\int^z_{h^{-1}(\operatorname{Re}h(z)+ib)}\omega_{\lambda,f}.$$
Notice that $\omega_{\lambda',\tilde{f}}=\omega_{\lambda,f}$, so \eqref{<01} together with Lemma \ref{lastlemma} shows that $H_1\in L^p_{\zeta_j}(\mathbb{D})$. And an argument identical to the one used for $H_2$ in the proof of Claim \ref{<j} shows that we also have $H_2\in L^p_{\zeta_0}(\mathbb{D})$ here. Thus we have
$$\frac{ \mathrm{e}^{\lambda h(z)}}{v(z)}\int_0^{z}\omega_{\lambda,f}\in A^p_{\zeta_0}.$$
Finally, the proof of Theorem \ref{eigenvalue} tells that $\frac{ \mathrm{e}^{\lambda h(z)}}{v(z)}\in A^p_{\zeta_0}.$ Consequently,
$$\frac{ \mathrm{e}^{\lambda h(z)}}{v(z)}\left(K-\int_0^{z}\omega_{\lambda,f}\right)\in A^p_{\zeta_0}$$
for any constant $K\in\mathbb{C}$. 
\end{proof}

Now we are at the point to give final inductions of the two main theorems.  In view of \eqref{difsolu} and Claim \ref{others}, the operator $\lambda-A$ is a surjection from $D(A)$ onto $A^p$ if and only if for each $f\in A^p$ there exists a constant $K$ such that
$$\frac{ \mathrm{e}^{\lambda h(z)}}{v(z)}\left(K-\int_0^z \omega_{\lambda,f}\right)\in A^p_{\zeta_j}$$
at every fixed points $\zeta_j$ ($j\geq 0$) of $(\varphi_t)$.

\begin{proof}[\textbf{Proof of Theorem \ref{main}}]
We rearrange (i)-(iii) of Theorem \ref{main} into the following five cases:

\noindent\emph{ Case \uppercase\expandafter{\romannumeral1}: $\gamma_1=-\infty$.} By Corollary \ref{rad} and Theorem \ref{eigenvalue}, we have
$$\sigma(A)=\lfloor-\infty, \gamma_0\rceil$$
in this case.

\noindent\emph{ Case \uppercase\expandafter{\romannumeral2}: $\gamma_0>\gamma_1>-\infty$.} Again by Theorem \ref{eigenvalue}, the strip $\lfloor\gamma_1, \gamma_0\rceil$ is contained in $\sigma(A).$

Suppose $\gamma_2<\operatorname{Re}\lambda<\gamma_1$. Claim \ref{>j}, \ref{<j}, and \ref{<0} show that for any $f\in A^p$,
$$\frac{ \mathrm{e}^{\lambda h(z)}}{v(z)}\int_z^{\zeta_1} \omega_{\lambda,f}\in A^p.$$
Hence, by taking $K=\int_0^{\zeta_1} \omega_{\lambda,f}$ in \eqref{difsolu}, we see that $\lambda-A$ is surjective, so $\lambda\notin\sigma(A)$.

Now suppose $\operatorname{Re}\lambda<\gamma_2$. Fix an arbitrary $f\in A^p$. If we assume that there existed a constant $K$ such that
$$\frac{ \mathrm{e}^{\lambda h(z)}}{v(z)}\left(K-\int_0^z \omega_{\lambda,f}\right)\in A^p,$$
then applying Claim \ref{<j} at $\zeta_1$, we see that the only possible choice of $K$ is $K=\int_0^{\zeta_1} \omega_{\lambda,f}$. But if we consider at  $\zeta_2$, then for the same reason we must have $K=\int_0^{\zeta_2} \omega_{\lambda,f}.$ Therefore, we have 
\begin{align}\label{linghua}
\int_{\zeta_1}^{\zeta_2} \omega_{\lambda,f}=0
\end{align}
for every $f$ in the range of $\lambda-A$. However, equation \eqref{linghua} cannot hold for all $f\in A^p$. In fact,  since $A^p$ is invariant under automorphisms, we may assume, after composing with a suitable automorphism, that $\zeta_1=-1$ and $\zeta_2=1$. Then by Remark \ref{zuihouremark}, the integral in \eqref{linghua} can be taken along the real segment $[-1,1]$. So equation \eqref{linghua} gives
$$\int_{-1}^{1} \mathrm{e}^{-\lambda h(x)}v(x)h'(x)f(x)\mathrm{d}x=0,$$
for all $f\in A^p$. By taking $f$ as polynomials, we have $\mathrm{e}^{-\lambda h(x)}v(x)h'(x)$ is identity zero on $(-1,1)$, which is  impossible. So $\lambda-A$ is not surjective, and therefore $\lfloor-\infty, \gamma_2\rceil\subset\sigma(A)$. To sum up, we have 
$$\sigma(A)=\lfloor-\infty, \gamma_2\rceil\cup\lfloor\gamma_1, \gamma_0\rceil$$
in this case.

\noindent\emph{ Case \uppercase\expandafter{\romannumeral3}: $\gamma_2<\gamma_0<\gamma_1$.} As shown in Case \uppercase\expandafter{\romannumeral2}, we have $\lambda\in\sigma(A)$ when $\operatorname{Re}\lambda\leq\gamma_2$, and $\lambda\notin\sigma(A)$ when $\gamma_2<\operatorname{Re}\lambda<\gamma_0$.

Now suppose $\gamma_0<\operatorname{Re}\lambda<\gamma_1$. Assume that $f$ belongs to the range of $\lambda-A$. Then according to Claim \ref{>0} and \ref{<j}, we must have
$$K=\int_0^{\zeta_0} \omega_{\lambda,f}=\int_0^{\zeta_1} \omega_{\lambda,f}$$
in \eqref{difsolu}. Consequently,
$$\int_{\zeta_0}^{\zeta_1} \omega_{\lambda,f}=0$$
for any $f$ in the range of $\lambda-A$. The same argument as in Case \uppercase\expandafter{\romannumeral2} then shows that $\lambda-A$ is not surjective, hence $\lambda\in\sigma(A)$. Therefore, we have 
$$\sigma(A)=\lfloor-\infty, \gamma_2\rceil\cup\lfloor\gamma_0, \gamma_1\rceil$$
in this case.

\noindent\emph{ Case \uppercase\expandafter{\romannumeral4}: $\gamma_0\leq\gamma_2$.} As shown in Case \uppercase\expandafter{\romannumeral2}  and Case \uppercase\expandafter{\romannumeral3} , both $\lfloor-\infty,\gamma_2\rceil$ and $\lfloor\gamma_0,\gamma_1\rceil$ are contained in $\sigma(A)$, hence by Corollary \ref{spectralbound}, 
$$\sigma(A)=\lfloor-\infty,\gamma_1\rceil.$$

\noindent\emph{ Case \uppercase\expandafter{\romannumeral5}.} The only remaining situation is when $\gamma_0=\gamma_1>\gamma_2$. As shown in Case \uppercase\expandafter{\romannumeral2} we have $\lfloor-\infty,\gamma_2\rceil\subset\sigma(A)$. Hence it suffices to prove that the line $\{\lambda\in\mathbb{C} : \operatorname{Re}\lambda=\gamma_0 \}$ is contained in $\sigma(A)$. 

Assume, for contradiction, that there exists some $\lambda$ on this line with $\lambda\notin\sigma(A)$. Then $\lambda-A$ is injective, which means that $\frac{\mathrm{e}^{\lambda h(z)}}{v(z)}$ does not belong to $A^p$. However, the proof of Theorem \ref{eigenvalue} shows that $\frac{\mathrm{e}^{\lambda h(z)}}{v(z)}\in A^p_\zeta$ for every $\zeta\in \partial \mathbb{D}\backslash\{\zeta_0, \zeta_1\}$. So either  $\frac{\mathrm{e}^{\lambda h(z)}}{v(z)}\notin A^p_{\zeta_0}$ or $\frac{\mathrm{e}^{\lambda h(z)}}{v(z)}\notin A^p_{\zeta_1}$. We will show that both possibilities lead to a contradiction.

First assume that  $\frac{\mathrm{e}^{\lambda h(z)}}{v(z)}\notin A^p_{\zeta_0}$. Set $\tilde{v}(z)=\frac{v(z)}{(z-\zeta_1)^c}$ and $\tilde{u}_t=\frac{\tilde{v}\circ\varphi_t}{\tilde{v}}$. Choose the constant $c$ so that $\gamma_2<\tilde{\gamma_1}<\gamma_0$. Then by Case \uppercase\expandafter{\romannumeral2} we have $\lambda\in\sigma(\tilde{A})$, where $\tilde{A}$ denotes the generator of $(\tilde{u}_tC_{\varphi_t})$. Since $v$ and $\tilde{v}$ are equivalent near $\zeta_0$, we have $\frac{\mathrm{e}^{\lambda h(z)}}{v(z)}\notin A^p_{\zeta_0}$, thus $\lambda-\tilde{A}$ is injective and consequently $\lambda-\tilde{A}$ is not surjective. On the other hand, since $\lambda-A$ is surjective, for every $f\in A^p$ there exists a constant $K_0\in\mathbb{C}$ such that
$$\frac{ \mathrm{e}^{\lambda h(z)}}{v(z)}\left(K_0-\int_0^z \omega_{\lambda,f}\right)$$
belongs to $A^p$, and in particular it belongs to $A^p_{\zeta_0}$.  The equivalence of $v$ and $\tilde{v}$ near $\zeta_0$ then implies
$$\frac{ \mathrm{e}^{\lambda h(z)}}{\tilde{v}(z)}\left(K-\int_0^z \tilde{\omega}_{\lambda,f}\right)\in A_{\zeta_0}^p$$
for some $K\in\mathbb{C}$, where 
$$\tilde{\omega}_{\lambda,f}=\mathrm{e}^{-\lambda h(\xi)}h'(\xi)\tilde{v}(\xi)f(\xi) \mathrm{d}\xi.$$
Applying Claims \ref{others} and \ref{>j} we deduce that
$$\frac{ \mathrm{e}^{\lambda h(z)}}{\tilde{v}(z)}\left(K-\int_0^z \tilde{\omega}_{\lambda,f}\right)\in A^p.$$
But this means that $\lambda-\tilde{A}$ is surjective, which is a contradiction.

If instead we have $\frac{\mathrm{e}^{\lambda h(z)}}{v(z)}\notin A^p_{\zeta_1}$, the argument is essentially the same: simply define $\tilde{v}(z)=\frac{v(z)}{(z-\zeta_0)^c}$ with a suitable $c$ so that $\gamma_2<\gamma_1<\tilde{\gamma}_0$ and then repeat the reasoning.

Combining all five cases we complete the proof of Theorem \ref{main}.
\end{proof}

\begin{proof}[\textbf{Proof of Theorem \ref{ess}}]

First note that for every $\lambda\in\mathbb{C}$, the solution \eqref{difsolu} implies that the kernel of $\lambda-A$ has dimension less than or equals to $1$.

Assume that $\operatorname{Re}\lambda\ne\gamma_j$ for all $j\geq0$. Then according to the proof of Theorem \ref{main}, the range of $\lambda-A$  falls into one of the following three cases:
\begin{enumerate}
\item If $\gamma_0<\operatorname{Re}\lambda$, the range is
$$\mathcal{L}_1=\bigcap_{\gamma_j>\operatorname{Re}\lambda}\left\{f\in A^p : \int_{\gamma_0}^{\gamma_j}\omega_{\lambda,f}=0\right\}.$$ 

\item If $\gamma_0>\operatorname{Re}\lambda$ and $\gamma_1>\operatorname{Re}\lambda$,  the range is
$$\mathcal{L}_2=\bigcap_{\substack{\gamma_j>\operatorname{Re}\lambda\\ j\geq 2}}\left\{f\in A^p : \int_{\gamma_1}^{\gamma_j}\omega_{\lambda,f}=0\right\}.$$ 

\item If  $\gamma_1<\operatorname{Re}\lambda<\gamma_0$, then $\lambda-A$ is surjective.
\end{enumerate}
The proofs of Lemma \ref{jifenshoulian} and \ref{zaibianyige} actually show that when $\gamma_0<\operatorname{Re}\lambda<\gamma_j$, the map
$$f\mapsto \int_{\gamma_0}^{\gamma_j}\omega_{\lambda,f}$$
defines a bounded linear functional on $A^p$. Consequently, $\mathcal{L}_1$ is the intersection of the kernels of finitely many bounded linear functionals, hence it is a closed subspace of $A^p$ with finite codimension. The same for $\mathcal{L}_2$. So $\lambda\notin\sigma_e(A)$ whenever $\operatorname{Re}\lambda\ne\gamma_j$ for all $j\geq0$.

Now suppose $\operatorname{Re}\lambda=\gamma_0$. Let 
$$\tilde{v}(z)=\frac{v(z)}{\Pi_{k\geq 1}(z-\zeta_k)^{c_k}},$$
and put $\tilde{u}_t=\frac{\tilde{v}\circ\varphi_t}{\tilde{v}}$. Take $\{c_k\}$ so that $\tilde{\gamma}_1<\gamma_0$. Let $\tilde{A}$ be the generator of $(\tilde{u}_tC_{\varphi_t})$. If $f\notin \mathrm{Ran}(\lambda-\tilde{A})$, then
$$\frac{ \mathrm{e}^{\lambda h(z)}}{\tilde{v}(z)}\left(K-\int_0^z \tilde{\omega}_{\lambda,f}\right)\notin A^p_{\zeta_0}$$
for any $K\in\mathbb{C}$, where $\tilde{\omega}_{\lambda,f}=\mathrm{e}^{-\lambda h(\xi)}h'(\xi)\tilde{v}(\xi)f(\xi) \mathrm{d}\xi.$ So
$$\frac{ \mathrm{e}^{\lambda h(z)}}{v(z)}\left(K-\int_0^z \omega_{\lambda,f}\right)\notin A^p_{\zeta_0}$$
for any $K\in\mathbb{C}$, thus $f\notin \mathrm{Ran}(\lambda-A)$. Therefore, we have $ \mathrm{Ran}(\lambda-A)\subset \mathrm{Ran}(\lambda-\tilde{A}).$ Assume that $\lambda\notin \sigma_e(A)$, then $\mathrm{Ran}(\lambda-A)$ is a closed subspace of $A^p$ with finite codimension. Then so is $\mathrm{Ran}(\lambda-\tilde{A})$. However, by Theorem \ref{main},
$$\sigma(\tilde{A})=\lfloor-\infty, \tilde{\gamma}_2\rceil\cup\lfloor \tilde{\gamma}_1,\gamma_0\rceil,$$
so $\lambda$ is a accumulation point of both the resolvent set and the spectrum of $A$. Thus $\lambda\in\sigma_e(\tilde{A})$ (see Theorem I.3.25 in \cite{book}), which is a contradiction.

Finally, if $\operatorname{Re}\lambda=\gamma_j$ for some $j\geq1$, then let  
$$\tilde{v}(z)=\frac{v(z)}{\Pi_{k\ne j}(z-\zeta_k)^{c_k}}$$ 
such that 
$$\tilde{\gamma}_2<\tilde{\gamma}_1=\gamma_j<\tilde{\gamma}_0.$$
Then we also have $ \mathrm{Ran}(\lambda-A)\subset \mathrm{Ran}(\lambda-\tilde{A}).$ By Theorem \ref{main}, $\lambda$ is the accumulation point of both the  resolvent set and the spectrum of $A$, so we have $\lambda\in\sigma_e(\tilde{A})$. Hence $\lambda\in\sigma_e(A)$.
\end{proof}

\section{Further remarks}\label{further}

We conclude this paper with some further remarks and open questions concerning our results.

First of all, as promised in Section \ref{sec2}, we give an example to show that behavior of $(u_tC_{\varphi_t})$ at non-fixed self-contact points of $(\varphi_t)$ can affect its spectral property.

\begin{example}\label{yigeyue}
 let $E=\{\pi/2^j : j=1, 2, \dots\}$ and 
$$\Omega_0=\{w\in\mathbb{C} : 0<\operatorname{Im}w<\pi\}\backslash \bigl((-\infty, -1]\times E\bigr).$$
For any fixed constant $d>0$, we can take real numbers $L_j\to-\infty$ such that
\begin{align}\label{gouzaobu}
\int_{(L_j, 0)\times(\pi/2^j, \pi/2^{j-1})}\left|\frac{1}{\mathrm{e}^{dw}+1}\right|^p\mathrm{d}v(w)\geq1
\end{align}
for every $j\geq 1$. Then consider the domain
$$\Omega=\Omega_0\big\backslash\cup_{j=1}^\infty(-\infty, L_j]\times(\pi/2^j, \pi/2^{j-1}).$$
Take $h$ to be a biholomorphic map from $\mathbb{D}$ onto $\Omega$ with $h(0)=0$, and define $(\varphi_t)$ to be the hyperbolic semiflow on $\mathbb{D}$ whose canonical model is $(\Omega, h)$.  Note that the only fixed point of $(\varphi_t)$ is its Denjoy-Wolff point, denoted by $\zeta_0$, with spectral value $\alpha_0=1$. Meanwhile $(\tilde{\varphi}_t)$ has one non-fixed self-contact point, denoted by $\eta$.

Now let
$$v(z)=G(z)^{\frac{2}{p}}\cdot\big(\mathrm{e}^{-dh(z)}+1\big),$$
where $G$ is the generator of $(\varphi_t)$, and define
$$u_t=\frac{v\circ\varphi_t}{v}=\left(\varphi_t'\right)^{\frac{2}{p}}\cdot\frac{\mathrm{e}^{dh}+\mathrm{e}^{-dt}}{\mathrm{e}^{dh}+1}.$$ 
Then $\beta_0=-2/p$, and hence $\gamma_0=0$.

Recall that $G=1/h'$, therefore
$$\int_{\mathbb{D}}\left|\frac{\mathrm{e}^{\lambda h(z)}}{v(z)}\right|^p\mathrm{d}v(z)=\int_\Omega\left|\frac{\mathrm{e}^{(\lambda +d)w}}{\mathrm{e}^{dw}+1}\right|^p\mathrm{d}v(w).$$ 
Thus by \eqref{gouzaobu}, when $\operatorname{Re}\lambda<-d$, one has $\frac{\mathrm{e}^{\lambda h}}{v}\notin A^p$. This means that $\lambda\notin\sigma_p(A)$ when $\operatorname{Re}\lambda<-d$.   On the other hand, if $-d<\operatorname{Re}\lambda<0$, we clearly have
$$\int_{\mathbb{D}}\left|\frac{\mathrm{e}^{\lambda h(z)}}{v(z)}\right|^p\mathrm{d}v(z)=\int_\Omega\left|\frac{\mathrm{e}^{(\lambda +d)w}}{\mathrm{e}^{dw}+1}\right|^p\mathrm{d}v(w)<\infty.$$
Therefore, for the generator $A$ of $(u_tC_{\varphi_t})$ we have
$$\overline{\sigma_p(A)}=\left\lfloor-d, 0\right\rceil\subsetneqq\left\lfloor-\infty, \gamma_0\right\rceil.$$

\end{example}

In this example, the semiflow $(\varphi_t)$ is analytic at its only fixed point $\zeta_0$, and $(u_t)$ is continuous at $\zeta_0$. However, the point spectrum of generator of $(u_tC_{\varphi_t})$ does \textbf{not} follow the description given in Theorem \ref{eigenvalue}. The only reason is that $\frac{\mathrm{e}^{\lambda \tilde{h}}}{\tilde{v}}$ fails to be integrable near the non-fixed self-contact point $\eta$ of $(\varphi_t)$. So this example dmonstrates that the spectra of hyperbolic weighted composition semigroups are influenced by their behavior near non‑fixed self‑contact points. Hence, we can ask the following question:
\\

\noindent\textbf{Question1.} How to describe the spectrum of $(u_tC_{\varphi_t})$ when $(\varphi_t)$ is not analytic at its self-contact points?
\\

Another issue of interest is the relation between the spectrum of the generator of a semigroup and the spectra of the individual operators in the semigroup. Following Section 3 of Chapter \uppercase\expandafter{\romannumeral4} in \cite{equ}, we say that a $C_0$-semigroup $(T_t)$ satisfies the \emph{spectral mapping theorem} (SMT) if
$$\sigma(T_t)\backslash\{0\} =  \mathrm{e}^{t\sigma(A)},$$
where $A$ denote the generator of $(T_t)$. And say that it satisfies the \emph{weak spectral mapping theorem} (WSMT) if
$$\sigma(T_t)\backslash\{0\} =\overline{  \mathrm{e}^{t\sigma(A)}}\backslash\{0\}.$$
While SMT or WSMT holds for many  $C_0$-semigroups, it fails for others. 

Theorem \ref{main} provides a description of the spectrum of the generator $A$ for a hyperbolic weighted composition $C_0$-semigroup $(u_tC_{\varphi_t})$. Then in view of the inclusion \eqref{SMT0},  when $\gamma_2\geq\min\{\gamma_0,\gamma_1\}$, the spectrum of $u_tC_{\varphi_t}$ is given by the following theorem. 

\begin{theorem}
Let $(u_tC_{\varphi_t})$ be a $C_0$-semigroup on $A^p$, where $(\varphi_t)$ is a hyperbolic semiflow. Suppose $(\varphi_t)$ is analytic at its self-contact points, and $(u_t)$ is continuous at these points. If  $\gamma_2\geq\gamma_0$ or $\gamma_2=\gamma_1$, then the spectrum of $u_tC_{\varphi_t}$ is the disk 
$$\bigl\{\lambda\in\mathbb{C} : |\lambda|\leq\max_{j\geq0}\mathrm{e}^{\gamma_jt}\bigr\}$$
for every $t>0$. 
\end{theorem}

Otherwise, if $\gamma_2<\min\{\gamma_0,\gamma_1\}$, we have
$$\sigma(u_tC_{\varphi_t})\supset\bigl\{\lambda : |\lambda|\leq\max_{j\geq0}\mathrm{e}^{\gamma_jt}\bigr\}\backslash\bigl\{\lambda : \mathrm{e}^{\gamma_2t}<|\lambda|<\min\{\mathrm{e}^{\gamma_0t},\mathrm{e}^{\gamma_1t}\}\bigr\}.$$
It is therefore natural to ask whether the intermediate annulus
$$\bigl\{\lambda : \mathrm{e}^{\gamma_2t}<|\lambda|<\min\{\mathrm{e}^{\gamma_0t},\mathrm{e}^{\gamma_1t}\}\bigr\}$$
is contained in the spectrum of $u_tC_{\varphi_t}$ for  $t>0$ in this case. That is, the following question:
\\

\noindent\textbf{Question2.} Is the spectrum of $u_tC_{\varphi_t}$ always the disk $\bigl\{\lambda\in\mathbb{C} : |\lambda|\leq\max_{j\geq0}\mathrm{e}^{\gamma_jt}\bigr\}?$
\\

If the answer is affirmative, it would provide an explicit example of a $C_0$-semigroup of weighted composition operators does not satisfy the SMT, nor even WSMT.
\\

Turning to eigenvalues, the spectral mapping theorem holds for the point spectra of all $C_0$-semigroups, see, e.g., p.277 of \cite{equ}. Consequently,
$$\sigma_p(u_tC_{\varphi_t})\backslash\{0\} =  \mathrm{e}^{t\sigma_p(A)}.$$
Since $0\notin \sigma_p(u_tC_{\varphi_t})$ for any $t\geq0$, we immediately obtain a description of the point spectra of the operators $u_tC_{\varphi_t}$ as a corollary of Theorem \ref{eigenvalue}.

\begin{corollary}
Let $(u_tC_{\varphi_t})$ is a $C_0$-semigroup on $A^p$, where $(\varphi_t)$ is a hyperbolic semiflow. Supose $(\varphi_t)$ is analytic at its self-contact points, and $(u_t)$ is continuous at these points.
 \begin{enumerate}
\item If $\gamma_1<\gamma_0$, then for any $t\geq0$,
$$\overline{\sigma_p(u_tC_{\varphi_t})}=\{\lambda\in\mathbb{C} :  \mathrm{e}^{\gamma_1t}\leq|\lambda|\leq \mathrm{e}^{\gamma_0t}\},$$
and
$$\sigma_p(u_tC_{\varphi_t})\supset\{\lambda\in\mathbb{C} :  \mathrm{e}^{\gamma_1t}<|\lambda|<\mathrm{e}^{\gamma_0t}\}.$$

\item If $\gamma_1>\gamma_0$, then for any $t>0$, we have $\sigma_p(u_tC_{\varphi_t})=\emptyset$.
\end{enumerate}
\end{corollary}

\end{document}